\documentclass[11pt,a4paper]{article}
\usepackage{amsmath}
\usepackage[a4paper]{geometry}
\usepackage{amssymb,latexsym,amsmath,amsfonts,amsthm}
\usepackage{graphicx}
\usepackage{eqnarray}
\usepackage{mathrsfs}
\usepackage{epsfig}
\usepackage{booktabs}

\usepackage{comment,verbatim}
\usepackage{overpic}
\usepackage{hyperref}
\usepackage{bbm}

\renewcommand{\Re}{\mathrm{Re}\,}
\renewcommand{\Im}{\mathrm{Im}\,}

\renewcommand{\vec}{\mathbf}

\newtheorem{theorem}{Theorem}[section]
\newtheorem{lemma}[theorem]{Lemma}

\newtheorem{corollary}[theorem]{Corollary}

\theoremstyle{definition}
\newtheorem{definition}[theorem]{Definition}

\theoremstyle{definition}
\newtheorem{example}[theorem]{Example}

\usepackage{color}

\theoremstyle{remark}

\newtheorem{remark}[theorem]{Remark}

\numberwithin{equation}{section}

\def\Xint#1{\mathchoice
{\XXint\displaystyle\textstyle{#1}}%
{\XXint\textstyle\scriptstyle{#1}}%
{\XXint\scriptstyle
\scriptscriptstyle{#1}}%
{\XXint\scriptscriptstyle
\scriptscriptstyle{#1}}%
\!\int}
\def\XXint#1#2#3{{
\setbox0=\hbox{$#1{#2#3}{\int}$}
\vcenter{\hbox{$#2#3$}}\kern-.5\wd0}}
\def\dashint{\Xint-}
\def\ddashint{\Xint=}

\begin{document}
\title{ Asymptotic expansions and fast computation of oscillatory
Hilbert transforms}
\author{Haiyong Wang\footnotemark[1]~\footnotemark[2], Lun Zhang\footnotemark[3] ~and
Daan Huybrechs\footnotemark[1]~\footnotemark[2]} \maketitle
\renewcommand{\thefootnote}{\fnsymbol{footnote}}
\footnotetext[1]{Department of Computer Science, Katholieke
Universiteit Leuven, Celestijnenlaan 200A, BE-3001 Leuven, Belgium
(why198309@yahoo.com.cn, haiyong.wang@cs.kuleuven.be,
daan.huybrechs@cs.kuleuven.be)}

\footnotetext[2]{This research was supported by the Fund for
Scientific Research -- Flanders through Research Project G.0617.10.}

\footnotetext[3]{Department of Mathematics, Katholieke Universiteit
Leuven, Celestijnenlaan 200B, BE-3001 Leuven, Belgium
(lun.zhang@wis.kuleuven.be) This author is a Postdoctoral Fellow of
the Fund for Scientific Research - Flanders (FWO), Belgium. }

\begin{abstract}
In this paper, we study the asymptotics and fast computation of the
one-sided oscillatory Hilbert transforms of the form
$$H^{+}(f(t)e^{i\omega t})(x)=\dashint_{0}^{\infty}e^{i\omega
t}\frac{f(t)}{t-x}dt,\qquad \omega>0,\qquad x\geq 0,
$$
where the bar indicates the Cauchy principal value and $f$ is a
real-valued function with analytic continuation in the first
quadrant, except possibly a branch point of algebraic type at the
origin. When $x=0$, the integral is interpreted as a Hadamard
finite-part integral, provided it is divergent. Asymptotic
expansions in inverse powers of $\omega$ are derived for each fixed
$x\geq 0$, which clarify the large $\omega$ behavior of this
transform. We then present efficient and affordable approaches for
numerical evaluation of such oscillatory transforms. Depending on
the position of $x$, we classify our discussion into three regimes,
namely, $x=\mathcal{O}(1)$ or $x\gg1$, $0<x\ll 1$ and $x=0$.
Numerical experiments show that the convergence of the proposed
methods greatly improve when the frequency $\omega$ increases. Some
extensions to oscillatory Hilbert transforms with Bessel oscillators
are briefly discussed as well.
\end{abstract}


\vspace{0.05in}


\section{Introduction}
Finite Fourier integrals of the form
\begin{equation}\label{eq:Fourier integrals}
\int_a^b f(x)e^{i\omega g(x)}dx
\end{equation}
with $\omega>0$ and $f(x)$, $g(x)$ being sufficiently smooth
functions have long been the subject of intensive study due to their
frequent occurrences in wide fields ranging from quantum chemistry,
image analysis, electrodynamics and computerized tomography to fluid
mechanics \cite{Iserles}. One difficulty in computing integrals
\eqref{eq:Fourier integrals} is that, for large frequency $\omega$,
the classical integration methods like Gauss quadrature are
inapplicable, since they often require many function evaluations
which make them highly time consuming. To overcome this difficulty,
many efficient approaches have been developed and significant
progress has occurred over the past few years. For instance, based
on asymptotic expansions of \eqref{eq:Fourier integrals} as $\omega$
tends to infinity, Iserles and N{\o}rsett \cite{Iserles,Iserles1}
proposed the asymptotic and Filon-type methods to evaluate
oscillatory integrals numerically. Other efficient approaches
include Levin-type methods, numerical steepest descent methods,
generalized quadrature rules, GMRES methods, modified
Clenshaw-Curtis methods, etc.; we refer to
\cite{asheim,deano,dominguez,evans,Huybrechs,levin,levin1,Olver,Olver1,Xiang,Xiang1}
and references therein for more information. All these methods
complement each other but share the advantageous property that their
accuracy improves greatly when $\omega$ increases.

Recently, oscillatory Hilbert transforms of the form
\begin{equation}\label{def:general transform}
\dashint_{\Gamma}e^{i\omega t}\frac{f(t)}{t-x}dt, \qquad x\in\Gamma,
\end{equation}
have received considerable attention as well. Here, $\Gamma$ is an
oriented curve in the complex plane, $f$ is a complex-valued
function satisfying a H\"{o}lder condition, and the bar denotes
Cauchy principal value. The interest in theoretical and numerical
study of such integral transforms arises from the fact that many
problems encountered in practice can be represented by an integral
equation with an oscillatory kernel having a singularity of Cauchy
type \cite{capobianco,Davis,king}; see also \cite{king book} for
numerous applications of Hilbert transforms in applied sciences.
Although the oscillatory Hilbert transforms \eqref{def:general
transform} bear some resemblances with \eqref{eq:Fourier integrals},
nevertheless, the singularity of Cauchy type suggests special
treatments. It will be especially interesting to see, as pointed out
in \cite{Olver 11}, if the aforementioned methods can be extended to
handle oscillatory Hilbert transforms.

For $\Gamma=[-1,1]$, we obtain from \eqref{def:general transform}
the finite oscillatory Hilbert transforms:
\begin{equation}\label{def:finite transform}
\dashint_{-1}^{1}e^{i\omega t}\frac{f(t)}{t-x}dt
=\lim_{\varepsilon\rightarrow0^{+}}\left(\int_{-1}^{x-\varepsilon}+\int_{x+\varepsilon}^{1}\right)e^{i\omega
t}\frac{f(t)}{t-x}dt
\end{equation}
with $-1<x<1$. An asymptotic expansion of \eqref{def:finite
transform} in inverse powers of $\omega$ was established by Lyness
in \cite{lyness} based on analytic continuation. Meanwhile, there
are several numerical schemes available to calculate
\eqref{def:finite transform}, most of which are typically based on
interpolatory type techniques. For example, Okecha \cite{Okecha1}
proposed to compute \eqref{def:finite transform} by using the
Lagrange interpolation polynomial of degree $n$ interpolating $f$ at
the $n+1$ zeros of the Legendre polynomial. Capobianco and Criscuolo
introduced a numerically stable procedure in \cite{capobianco},
which relies on an interpolatory procedure at the zeros of Jacobi
polynomials. In a recent paper \cite{wang1}, Wang and Xiang have
presented an integration rule of interpolatory type with the aid of
the Chebyshev points of the second kind. The rule is uniformly
convergent with respect to the pole $x$ when $f$ is analytic in a
neighborhood of the interval $[-1,1]$, and it can be implemented by
means of the fast Fourier transform (FFT). If the function $f$ is
analytic in a sufficiently large region of the complex plane
containing $[-1,1]$, then the complex integration method \cite{Mil}
and the numerical steepest descent method \cite{Huybrechs} can be
extended to compute such integrals efficiently, and the accuracy
improves greatly as $\omega$ increases; see \cite{wang2} for
details.

In this paper, we are concerned with one-sided oscillatory Hilbert
transforms on the positive real axis:
\begin{align}\label{def:one-side transform}
H^+(f(t)e^{i\omega t})(x):= \dashint_{0}^{\infty}e^{i\omega
t}\frac{f(t)}{t-x}dt =
\lim_{\varepsilon\rightarrow0^{+}}\left(\int_{0}^{x-\varepsilon}+\int_{x+\varepsilon}^{\infty}\right)
e^{i\omega t}\frac{f(t)}{t-x}dt,\quad x\geq 0,
\end{align}
i.e., $\Gamma=\mathbb{R}^+$ in \eqref{def:general transform}. Here,
$f$ is a real-valued function satisfying some conditions. In
particular, it has an analytic continuation in the first quadrant of
the complex plane, except possibly a branch point of algebraic type
at the origin. When $x=0$, the integral is interpreted as a Hadamard
finite-part integral, provided it is divergent; see Section
\ref{sec:hadamard} below for a definition. We point out that the
integral \eqref{def:one-side transform} is also closely related to
infinite oscillatory Hilbert transforms on the real axis given by
\begin{equation*}\label{def:Hilbert trans on real line}
H(f(t)e^{i\omega t})(x):=\dashint_{-\infty}^{\infty}e^{i\omega
t}\frac{f(t)}{t-x}dt =
\lim_{\varepsilon\rightarrow0^{+}}\left(\int_{-\infty}^{x-\varepsilon}+\int_{x+\varepsilon}^{\infty}\right)
e^{i\omega t}\frac{f(t)}{t-x}dt, \qquad x\in\mathbb{R}.
\end{equation*}
Indeed, by assuming $x>0$, it is easily seen that
\begin{equation}\label{eq:splitting}
H(f(t)e^{i\omega t})(x)=\int_{0}^{\infty}-e^{-i\omega
t}\frac{f(-t)}{t+x}dt+H^+(f(t)e^{i\omega t})(x).
\end{equation}
The first integral on the right hand side of \eqref{eq:splitting} is
the Stieltjes transform of $-e^{-i\omega t}f(-t)$, which is a
regular integral for $x>0$.

A large $x$ expansion of one-sided oscillatory Hilbert transform was
already established by Wong \cite{wong}. Instead of
\eqref{def:one-side transform}, he considered $H^+(f(t))(x)$.
However, it is assumed that $f$ is a locally integrable function on
$[0,\infty)$ and has an asymptotic expansion of the form
$$
f(t)\sim e^{ict}\sum_{s=0}^{\infty}a_st^{-s-\alpha}, \quad
\mathrm{as}\quad t\rightarrow\infty,
$$
where $0<\alpha\leq1$ and $c$ is a real number. Thus, one may have
$c=\omega$. Let $\psi_0(t)=f(t)$ and define $\psi_n(t)$ by
$$
f(t)=\sum_{s=0}^{n-1}a_se^{ic t}t^{-s-\alpha}+\psi_n(t), \qquad
n\geq1.
$$
It was then shown that (see \cite[Thm.~1]{wong})
\begin{equation}\label{eq:large x expansion}
H^{+}(f(t))(x)=E_{\alpha,c}(x)\sum_{s=0}^{n-1}\frac{a_s}{x^s}
-\sum_{s=1}^{n}\frac{b_s}{x^s}+\frac{1}{x^n}\delta_n(x), \quad n\geq
1,
\end{equation}
for $0<\alpha<1$, where
$$
E_{\alpha,c}(x)=\frac{e^{ic
x}}{x^\alpha}\left[e^{-i\alpha\pi}\Gamma(1-\alpha)\Gamma(\alpha,ic
x)+i\pi \right],
$$
$$
b_s=\int_{0}^{\infty}t^{s-1}\psi_s(t)dt, \quad s\geq1,
$$
and
$$
\delta_n(x)=\dashint_{0}^{\infty}\frac{t^n\psi_n(t)}{t-x}dt, \quad
n=0,1,\ldots.
$$
Here $\Gamma(\alpha)$ is the Gamma function \cite[p.~255]{Abram} and
$\Gamma(\alpha,z)$ is the complementary incomplete Gamma function
\cite[p.~260]{Abram}. The expansions when $\alpha=1$ are also
derived in a similar manner; see \cite[Thm.~2 and Thm.~3]{wong}.
Moreover, the following bounds for $\delta_n(x)$ are achieved:
\begin{equation*}
|\delta_n(x)|\leq M_n\frac{\ln x}{x^{\alpha}},
\end{equation*}
for all $x>e$, where $M_n$ is a positive constant. The difficulty in
applying expansions \eqref{eq:large x expansion} is that, as also
pointed out by Wong, the coefficients $b_s$ are inconvenient for
calculations. In a later paper \cite{ursell}, Ursell generalized the
results of Wong and showed that these coefficients can be readily
determined whenever the Mellin transform of $f(t)$ is known. An
interesting example was given for $f(t)=J_0^2(t)$ with applications
in water waves, where $J_0(t)$ is the zeroth-order Bessel function
of the first kind.

For the numerical aspects of \eqref{def:one-side transform}, King et
al. \cite{king} constructed a fairly robust numerical procedure by
using convergence accelerator techniques. Unfortunately, this series
acceleration method may suffer from difficulties when the
singularity is embedded in a region of extreme oscillatory behavior.

The purpose of this paper is two-fold. On the one hand, we shall
derive asymptotic expansions of one-sided oscillatory Hilbert
transforms \eqref{def:one-side transform} as $\omega\to\infty$. To
the best of our knowledge, none of the studies are available in this
direction. Such an expansion clarifies the behavior of
\eqref{def:one-side transform} for large $\omega$ and also provides
a powerful mean for the design of effective computational methods.
On the other hand, in view of the fact that asymptotic expansions
are not suitable for numerical calculation, we present efficient
quadrature rules to approximate such integrals. It comes out that
these rules depend on the position of $x$. This can be seen from
\eqref{def:one-side transform} and \eqref{eq:large x expansion},
where the integral may tend to zero as $x\rightarrow\infty$ and blow
up as $x\to 0$.

The rest of this paper is organized as follows. We perform
asymptotic analysis of oscillatory Hilbert transforms in Section
\ref{sec:asy}. The analyticity of $f$ is of importance in our
derivation. In Section \ref{sec:computation}, we propose efficient
and affordable approaches for numerical evaluation of such
oscillatory transforms. These methods are designed for three
regimes, that is, $x=\mathcal{O}(1)$ or $x\gg1$, $0<x\ll 1$ and
$x=0$, which cover all the situations. Numerical experiments show
that the convergence of the proposed methods greatly improves when
the frequency $\omega$ increases. Some ideas in this paper can also
be extended to study oscillatory Hilbert transforms with Bessel
oscillators. We give a brief description of this aspect in Section
\ref{sec:extension}. We conclude this paper with some final remarks
in Section \ref{sec:remarks}.

\section{Asymptotic analysis of oscillatory Hilbert
transforms}\label{sec:asy}

\subsection{Large $\omega$ expansion with $x>0$}
We start with the derivation of asymptotic expansions of oscillatory
Hilbert transforms \eqref{def:one-side transform} for large $\omega$
with $x>0$. An important ingredient in our analysis is the following
lemma which allows us to reduce the Cauchy principal integrals
\eqref{def:one-side transform} to ordinary integrals under certain
restrictions on $f$.

\begin{lemma}\label{lem:transfer transform}
Let $f$ be a locally integrable function on $[0,\infty)$ and
continuously differentiable over $(0, \infty)$. Suppose that $f$ has
an analytic continuation in the first quadrant of the complex plane,
except possibly a branch point at the origin, and there exist
constants $M>0$, $\delta<1$ and $0 \leq d < \omega$ such that
\begin{equation}\label{eq:growth condition}
|f(z)|\leq M|z|^{\delta}e^{d\Im{(z)}},
\end{equation}\label{eq:growth of f}
as $|z|\rightarrow\infty$ in the first quadrant. Then we have
\begin{equation}\label{eq:tranfer trans}
H^+(f(t)e^{i\omega t})(x)= \dashint_{0}^{\infty}e^{i\omega
t}\frac{f(t)}{t-x}dt =i\pi e^{i\omega x}
f(x)+\int_{0}^{\infty}e^{-\omega p}\frac{f(ip)}{p+ix}dp
\end{equation}
for each $x>0$, whenever the integral exists.
\end{lemma}
\begin{proof}
Let us consider a quarter of the disc centered at the origin with
radius $R$, which lies in the first quadrant, and denote by
$\Gamma_R$ its boundary, i.e.,
\begin{equation}
\Gamma_R:=\{ z \,|\, |z|=R, \Re (z) >0, \Im (z)>0 \}.
\end{equation}
For each fixed $x>0$, we can find $R$ large enough such that a half
disc $U_{x,\epsilon}^+:=\{ z \,|\, |z-x|\leq \epsilon, \Im (z)>0 \}$
can be excluded from the quarter of the disc, where $\epsilon$ is a
small positive number. The obtained domain is bounded by curves
orientated in a counter-clockwise manner as illustrated in Figure
\ref{fig:integralcontour}.
\begin{figure}[h]
\centering
\begin{overpic}[scale=0.7]{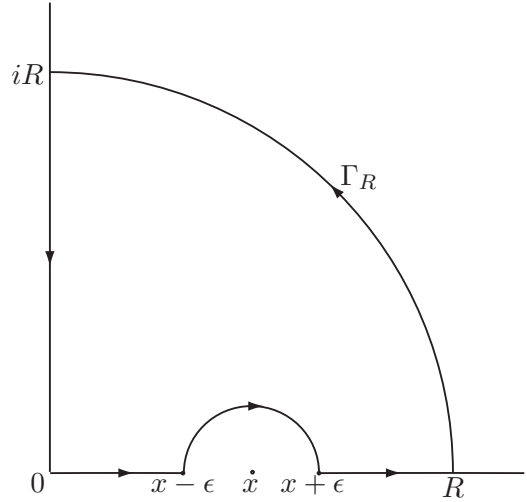}
\put(-7,82.5){$iR$} \put(-3,-3){0} \put(22,-3){$x-\epsilon$}
\put(41,-3){$x$} \put(49,-3){$x+\epsilon$} \put(82.5,-4){$R$}
\put(61.5,61.5){$\Gamma_R$}
\end{overpic}
\caption{A quarter disc in the first quadrant with a small
indentation at $x$. } \label{fig:integralcontour}
\end{figure}

According to our assumptions on $f$, it is easily seen that the
integrand $e^{i \omega t}\frac{f(t)}{t-x}$ of \eqref{def:one-side
transform} is analytic in a quarter of the disc with a small
indentation at $x$, as described above. We then obtain from Cauchy's
theorem that
\begin{equation}
\left( \int_0^{x-\epsilon}+\int_{\partial U_{x,\epsilon}^+}
+\int_{x+\epsilon}^R+\int_{\Gamma_R}+\int_{i R}^0\right) e^{i \omega
t}\frac{f(t)}{t-x} dt=0,
\end{equation}
where $\partial U_{x,\epsilon}^+$ stands for the boundary of
$U_{x,\epsilon}^+$. This, together with \eqref{def:one-side
transform}, implies
\begin{equation}\label{eq:three integrals}
H^+(f(t)e^{i\omega t})(x)=\lim_{\epsilon \to 0, R \to \infty}\left(
-\int_{\Gamma_R}+\int_{0}^{i R}-\int_{\partial
U_{x,\epsilon}^+}\right) e^{i \omega t}\frac{f(t)}{t-x} dt.
\end{equation}
We next evaluate the three integrals on the right hand side of
\eqref{eq:three integrals}.

A simple change of variable
$$
t=Re^{i\theta},\qquad 0\leq\theta\leq\frac{\pi}{2},
$$
yields
\begin{align}
\left|\int_{\Gamma_R}e^{i \omega t}\frac{f(t)}{t-x} dt\right|
&=\left|\int_{0}^{\frac{\pi}{2}}e^{i\omega Re^{i\theta}}
\frac{f(Re^{i\theta})}{Re^{i\theta}-x}Re^{i\theta}id\theta \right| \nonumber \\
&\leq R\int_{0}^{\frac{\pi}{2}}e^{-\omega
R\sin\theta}\left|\frac{f(Re^{i\theta})}{Re^{i\theta}-x}\right|d\theta \nonumber \\
&\leq \frac{R}{R-x}\int_{0}^{\frac{\pi}{2}}e^{-\omega
R\sin\theta}|f(Re^{i\theta})|d\theta \nonumber
\\
&\leq
\frac{MR^{1+\delta}}{R-x}\int_{0}^{\frac{\pi}{2}}e^{-(\omega-d)
R\sin\theta}d\theta, \label{eq:int on circle}
\end{align}
where in the last step we have made use of \eqref{eq:growth of f}.
Recall the well-known inequality $ \sin\theta \geq
\frac{2}{\pi}\theta$, if $0\leq\theta \leq\frac{\pi}{2}$; cf.
\cite[p. 223]{Ablowitz}. For $R$ large enough, we obtain
\begin{align}\label{eq:est of 1st int}
\left|\int_{\Gamma_R}e^{i \omega t}\frac{f(t)}{t-x} dt\right| &\leq
\frac{MR^{1+\delta}}{R-x}\int_{0}^{\frac{\pi}{2}}e^{-\frac{2}{\pi}(\omega-d)
R\theta}d\theta, \nonumber \\
&=\frac{\pi M R^\delta}{2(\omega-d) (R-x)}\left(1-e^{-(\omega-d)
R}\right)\to 0, \qquad \textrm{as $R\to\infty$}.
\end{align}
It is also easily seen that
\begin{align}
\lim_{R\to\infty}\int_{0}^{iR}e^{i \omega t}\frac{f(t)}{t-x} dt
&=\lim_{R\to\infty}\int_{0}^{R}e^{- \omega p}\frac{f(ip)}{ip-x}idp \nonumber \\
&=\int_{0}^{\infty}e^{- \omega p}\frac{f(ip)}{p+ix}dp.
\end{align}
To evaluate the third integral over the contour $\partial
U_{x,\epsilon}^+$, we appeal to \cite[(2.7) and (2.9)]{wang2}, which
gives
\begin{equation}\label{eq:est of 3rd int}
\lim_{\epsilon\rightarrow0}\int_{\partial
U_{x,\epsilon}^+}e^{i\omega t}\frac{f(t)}{t-x}dt=-i\pi e^{i\omega
x}f(x).
\end{equation}

Finally, substituting \eqref{eq:est of 1st int}--\eqref{eq:est of
3rd int} into \eqref{eq:three integrals}, we obtain
\eqref{eq:tranfer trans}.
\end{proof}

Now, we are ready to prove
\begin{theorem}\label{thm:asy x>0}
Let $f$ be a function as given in Lemma \ref{lem:transfer transform}
and assume that $f$ takes an asymptotic expansion of the form
\begin{equation}\label{eq:asy of f}
f(t)\sim\sum_{j=0}^{\infty}a_jt^{j-\alpha}
\end{equation}
as $t\to 0^+$, where $0 \leq \alpha <1$. Then the one-sided
oscillatory Hilbert transforms \eqref{def:one-side transform} can be
expanded in the following fashion
\begin{equation}\label{eq:asy x>0}
H^+(f(t)e^{i\omega t})(x)\sim i\pi e^{i\omega
x}f(x)-\sum_{\ell=0}^{\infty}\frac{\Gamma(\ell+1-\alpha)}{\omega^{\ell+1-\alpha}}
e^{\frac{\pi}{2}(\ell+1-\alpha)i}\left(\sum_{j+k=\ell}\frac{a_j}{x^{k+1}}\right)
\end{equation}
as $\omega\to\infty$, for each $x>0$.
\end{theorem}
\begin{proof}
By \eqref{eq:asy of f}, it follows that, for each fixed $x>0$,
\begin{align}
\frac{f(ip)}{p+ix} & \sim \frac{1}{ix}\left(\sum_{j=0}^{\infty}
a_j(ip)^{j-\alpha}\right)
\left(\sum_{k=0}^{\infty}\frac{(ip)^k}{x^k}\right)
\nonumber \\
&=-i\sum_{\ell=0}^{\infty}
i^{\ell-\alpha}p^{\ell-\alpha}\left(\sum_{j+k=\ell}\frac{a_j}{x^{k+1}}\right)
\nonumber\\
&= -\sum_{\ell=0}^{\infty}
e^{\frac{\pi}{2}(\ell+1-\alpha)i}p^{\ell-\alpha}
\left(\sum_{j+k=\ell}\frac{a_j}{x^{k+1}}\right)
\end{align}
as $p\to 0^+$. In view of \eqref{eq:tranfer trans}, an appeal to
Watson's lemma (cf. \cite[p. 20]{wong1}) gives us,
\begin{align*}
H^+(f(t)e^{i\omega t})(x)&=i\pi e^{i\omega x}
f(x)+\int_{0}^{\infty}e^{-\omega p}\frac{f(ip)}{p+ix}dp \\
&\sim  i\pi e^{i\omega x}
f(x)-\sum_{\ell=0}^{\infty}e^{\frac{\pi}{2}(\ell+1-\alpha)i}
\left(\sum_{j+k=\ell}\frac{a_j}{x^{k+1}}\right)
\int_{0}^{\infty}p^{\ell-\alpha}e^{-\omega p}dp\\
&= i\pi e^{i\omega x}
f(x)-\sum_{\ell=0}^{\infty}\frac{\Gamma(\ell+1-\alpha)}{\omega^{\ell+1-\alpha}}
e^{\frac{\pi}{2}(\ell+1-\alpha)i}
\left(\sum_{j+k=\ell}\frac{a_j}{x^{k+1}}\right),
\end{align*}
as $\omega\to\infty$, which is \eqref{eq:asy x>0}.
\end{proof}

\begin{remark}
It is worth noting that the expansion \eqref{eq:asy x>0} is only
uniformly valid for $x$ bounded away from $0$. To clarify the
behavior when $0<x\ll1$, we make the following decomposition:
\begin{align}\label{eq:decomp x near 0}
\dashint_{0}^{\infty}\frac{f(t)}{t-x}e^{i\omega t}dt &
=\dashint_{0}^{\infty}\frac{f_{\alpha}(t)}{t^{\alpha}(t-x)}e^{i\omega
t}dt \nonumber \\
& =
\int_{0}^{\infty}\frac{f_{\alpha}(t)-f_{\alpha}(x)}{t^{\alpha}(t-x)}e^{i\omega
t}dt + f_{\alpha}(x)\dashint_{0}^{\infty}\frac{e^{i\omega
t}}{t^{\alpha}(t-x)}dt
\end{align}
with $f_{\alpha}(t)=t^{\alpha}f(t)$. For the first integral on the
right hand side of \eqref{eq:decomp x near 0}, in view of the
expansion
\[
\frac{f_{\alpha}(t)-f_{\alpha}(x)}{t-x} \sim
\left(\sum_{j=1}^{\infty}a_jx^{j-1}\right) +
\left(\sum_{j=2}^{\infty}a_jx^{j-1}\right)t + \mathcal{O}(t^2)
\]
as $t\rightarrow0^{+}$, then by \cite[Thm.~1, p.~199]{wong} we find
that
\[
\int_{0}^{\infty}\frac{f_{\alpha}(t)-f_{\alpha}(x)}{t^{\alpha}(t-x)}e^{i\omega
t}dt \sim
\left(\sum_{j=1}^{\infty}a_jx^{j-1}\right)\frac{\Gamma(1-\alpha)e^{\frac{\pi}{2}(1-\alpha)i}}{\omega^{1-\alpha}}
+ \mathcal{O}\left(\frac{1}{\omega^{2-\alpha}}\right), \quad
\omega\rightarrow\infty,
\]
where the sum on the right side converges when $0<x\ll1$. For the
second one, we obtain from \cite{wong} and \eqref{eq:c=0} below that
\begin{align}\label{eq:exact cpv form}
\dashint_{0}^{\infty}\frac{e^{i\omega t}}{t^{\alpha}(t-x)}dt =
\left\{
                    \begin{array}{ll}
                      \frac{e^{i\omega
x}}{x^\alpha}\left[e^{-i\alpha\pi}\Gamma(1-\alpha)\Gamma(\alpha,i\omega
x)+i\pi \right], & \hbox{if $0<\alpha<1$,} \\
                     e^{i\omega x} [ \mathrm{Ei}(1,i\omega x)+i \pi], & \hbox{if $\alpha=0$.}
                    \end{array}
                  \right.
\end{align}
Combining these results, we can see clearly that the asymptotic
behaviour of the oscillatory Hilbert transforms \eqref{def:one-side
transform} depends strongly on the behaviour of the product of
$\omega$ and $x$. For example, the asymptotic behaviour of
\eqref{def:one-side transform} as $x\rightarrow0^{+}$,
$\omega\rightarrow\infty$ and $\omega x\rightarrow0^{+}$ can be
derived by taking into account the asymptotic expansions of
$\mathrm{Ei}(1,i\omega x)$ and $\Gamma(\alpha,i\omega x)$
respectively. We omit the details here and only give a leading term
\begin{align}\label{eq:exact cpv form}
\dashint_{0}^{\infty}\frac{f(t)}{t-x}e^{i\omega t}dt \sim \left\{
                    \begin{array}{ll}
                       \frac{a_0}{x^{\alpha}}\left[ \frac{\pi}{\sin(\alpha\pi)}e^{-i\alpha\pi}+i\pi \right] + \mathcal{O}(\omega^{\alpha}), & \hbox{if $0<\alpha<1$,}  \\
                     - a_0 \log(\omega x) + \mathcal{O}(1), & \hbox{if $\alpha=0$.}
                    \end{array}
                  \right.
\end{align}


\end{remark}

\begin{remark}
In Theorem \ref{thm:asy x>0}, we require that $f$ has an analytic
continuation in the first quadrant and satisfies the growth
condition \eqref{eq:growth of f}. These conditions can be further
relaxed to allow $f$ has a simple pole in the first quadrant or $f$
only has an analytic continuation around the real axis and satisfies
\eqref{eq:growth of f}. In both cases, the only contribution will be
a exponentially small term in $\omega$, thus, the expansion
\eqref{eq:asy x>0} still holds.
\end{remark}

\subsection{Large $\omega$ expansion with $x=0$: asymptotics of Hadamard finite-part integrals}
\label{sec:hadamard}

When $x=0$, the integrand of the one-sided oscillatory Hilbert
transforms \eqref{def:one-side transform} has a singularity at the
origin. If the integral is divergent, the transform should be
understood as a finite-part integral in the Hadamard sense. Note
that the definition of Hadamard finite-part integral does not change
the value of a convergent integral. It is the aim of this section to
find the asymptotics of \eqref{def:one-side transform} with $x=0$.

Assume that $f$ still admits an asymptotic expansion near the origin
as given in \eqref{eq:asy of f}, we then formally have
\begin{equation}\label{eq:subtracting singularity}
\int_0^{\infty}e^{i\omega
t}\frac{f(t)}{t}dt=a_0\int_0^{\infty}\frac{e^{i\omega
t}}{t^{\alpha+1}}dt+\int_0^{\infty}e^{i\omega
t}\frac{f(t)-a_0t^{-\alpha}}{t}dt.
\end{equation}
There are two integrals on the right hand side of
\eqref{eq:subtracting singularity}. The first one is divergent and
should be interpreted as a Hadamard finite-part integral over the
positive real axis. The integrand of the second one has an
integrable singularity at the origin. Hence, it is well defined.

To this end, one needs to extend standard Hadamard finite-part
integrals for the finite interval (cf. \cite[Sec. 1.4]{Krommer}) to
semifinite integrals. Here, we adapt the definition from
\cite{monegato}:
\begin{definition}\label{def:finite part}
Let $g(x)$ be of class $C^{m+1}[0,\infty)$ and such that
\begin{equation}\label{eq:allowable}
\left|\int_{0}^{\infty}g^{(k)}(t)t^{p-1}dt\right|<\infty, \quad
k=0,\ldots,m+1,
\end{equation}
for all $p\geq1$. Then for any $\eta\geq1$, a finite-part integral
of order $\delta$ for the positive real axis is defined as
\begin{equation}\label{eq:Hadamard definition}
\ddashint_{0}^{\infty}\frac{g(t)}{t^{\eta}}dt:=
\ddashint_{0}^{b}\frac{g(t)}{t^{\eta}}dt+\int_{b}^{\infty}\frac{g(t)}{t^{\eta}}dt,
\end{equation}
where $\ddashint$ stands for a Hadamard finite-part integral and $b$
is an arbitrary positive constant.
\end{definition}

In \cite{monegato}, a function $g$ satisfying \eqref{eq:allowable}
is called allowable. The Definition \ref{def:finite part} defines a
Hadamard finite-part integral for the infinite interval from a
finite-part integral on the finite interval, and it is independent
of the choice of $b$.

With Definition \ref{def:finite part}, we observe that
\begin{lemma}\label{lem:calculation of finite part integrals}
\begin{align}\label{eq:calculation of finite part integrals}
\ddashint_0^{\infty}\frac{e^{i\omega t}}{t^{\alpha+1}}dt=\left\{
                    \begin{array}{ll}
                      \frac{e^{\frac{\pi}{2}(2-\alpha)i}\omega^\alpha}{\alpha}\Gamma(1-\alpha), & \hbox{if $0<\alpha<1$,} \\
                      -\gamma-\log\omega+i\frac{\pi}{2}, & \hbox{if $\alpha=0$,}
                    \end{array}
                  \right.
\end{align}
where $\gamma$ is the Euler constant.
\end{lemma}
\begin{proof}
From \eqref{eq:Hadamard definition}, it is readily seen that
\begin{align}\label{def:fp}
\ddashint_{0}^{\infty}\frac{e^{i\omega t}}{t^{\alpha+1}}dt:=
\ddashint_{0}^{b}\frac{e^{i\omega
t}}{t^{\alpha+1}}dt+\int_{b}^{\infty}\frac{e^{i\omega
t}}{t^{\alpha+1}}dt,
\end{align}
where $b$ is an arbitrary positive constant. If $0<\alpha<1$, an
appeal to integration by parts gives us
\begin{align}\label{def:fp part1 without limit}
\int_{\varepsilon}^{b}\frac{e^{i\omega t}}{t^{\alpha+1}}dt
&=\frac{e^{i\omega
\varepsilon}}{\alpha\varepsilon^{\alpha}}-\frac{e^{i\omega
b}}{\alpha
b^{\alpha}}+\frac{i\omega}{\alpha}\int_{\varepsilon}^{b}t^{-\alpha}e^{i\omega
t}dt \nonumber \\
&=\frac{e^{i\omega
\varepsilon}}{\alpha\varepsilon^{\alpha}}-\frac{e^{i\omega
b}}{\alpha
b^{\alpha}}+\frac{i\omega}{\alpha}\left(\frac{\Gamma(1-\alpha)}{(-i\omega)^{1-\alpha}}-b^{1-\alpha}\mathrm{Ei}(\alpha,-i\omega
b)-\int_{0}^{\varepsilon}t^{-\alpha}e^{i\omega t}dt\right),
\end{align}
and
\begin{align}\label{def:fp part2}
\int_{b}^{\infty}\frac{e^{i\omega
t}}{t^{\alpha+1}}dt=\frac{e^{i\omega b}}{\alpha
b^{\alpha}}+\frac{i\omega}{\alpha}\int_{b}^{\infty}t^{-\alpha}e^{i\omega
t}dt=\frac{e^{i\omega b}}{\alpha
b^{\alpha}}+\frac{i\omega}{\alpha}b^{1-\alpha}\mathrm{Ei}(\alpha,-i\omega
b),
\end{align}
where $\varepsilon$ is a small positive number and
\begin{equation}
\mathrm{Ei}(\rho,z):=\int_1^{\infty}t^{-\rho }e^{-zt}dt,\quad
\rho>0,\quad \Re(z)\geq 0,
\end{equation}
is the exponential integral. In \eqref{def:fp part1 without limit},
by neglecting the divergent term $\frac{e^{i\omega
\varepsilon}}{\alpha\varepsilon^{\alpha}}$ and noting that the last
integral vanishes as $\varepsilon\rightarrow0^+$, we obtain
\begin{align}\label{def: fp part11}
\ddashint_{0}^{b}\frac{e^{i\omega
t}}{t^{\alpha+1}}dt:=-\frac{e^{i\omega b}}{\alpha
b^{\alpha}}+\frac{i\omega}{\alpha}\left(\frac{\Gamma(1-\alpha)}{(-i\omega)^{1-\alpha}}-b^{1-\alpha}\mathrm{Ei}(\alpha,-i\omega
b)\right).
\end{align}
Combining \eqref{def:fp}, \eqref{def: fp part11} and \eqref{def:fp
part2}, we get the desired result.

Similarly, if $\alpha=0$, we note that
\begin{align}
\left(\int_{\varepsilon}^{b}+\int_b^{\infty}\right)\frac{e^{i\omega
t}}{t}dt=\mathrm{Ei}(1,-i\omega
\varepsilon)=-\gamma-\log\omega+i\frac{\pi}{2}-\log\varepsilon+\mathcal{O}(\varepsilon),
\end{align}
as $\varepsilon\to 0$, where $\gamma$ is the Euler constant. Hence,
\begin{align}\label{def:fp alpha=0}
\ddashint_0^\infty\frac{e^{i\omega
t}}{t}dt:=-\gamma-\log\omega+i\frac{\pi}{2},
\end{align}
as shown in \eqref{eq:calculation of finite part integrals}.
\end{proof}

A combination of Lemma \ref{lem:calculation of finite part
integrals} and \eqref{eq:subtracting singularity} gives us
\begin{equation}\label{eq:Hadamard int}
\ddashint_0^{\infty}e^{i\omega t}\frac{f(t)}{t}dt=\left\{
                    \begin{array}{ll}
                      \frac{e^{\frac{\pi}{2}(2-\alpha)i}\omega^\alpha}{\alpha}\Gamma(1-\alpha)a_0+\int_0^{\infty}e^{i\omega
t}\frac{f(t)-a_0t^{-\alpha}}{t}dt, & \hbox{if $0<\alpha<1$,} \\
                      (-\gamma-\log\omega+i\frac{\pi}{2})f(0)+\int_0^{\infty}e^{i\omega
t}\frac{f(t)-f(0)}{t}dt, & \hbox{if $\alpha=0$.}
                    \end{array}
                  \right.
\end{equation}
Here we have made use of the fact that $a_0=f(0)$ if $\alpha=0$ in
\eqref{eq:asy of f}. To derive asymptotic expansions of Hadamard
finite part integrals \eqref{eq:Hadamard int}, it is then sufficient
to find asymptotics of Fourier-type integrals $\int_0^{\infty}\tilde
f(t)e^{i\omega t}dt$, where
\begin{equation}\label{def:g}
\tilde f(t):=\frac{f(t)-a_0t^{-\alpha}}{t},\qquad 0\leq\alpha<1.
\end{equation}
Our result is stated below:
\begin{theorem}\label{thm:asy of Hadamard}
Let $f$ be a locally integrable function on $[0,\infty)$ and $m$
times continuously differentiable over $(0, \infty)$, $m$ being a
positive integer. Suppose that $f$ has the asymptotic expansion
\eqref{eq:asy of f}, and this expansion can be differentiated $m$
times. Moreover, each of the integrals
\begin{equation}
\int_{1}^{\infty}\tilde f^{(s)}(t)e^{i\omega t}dt, \qquad
s=0,1,\cdots,m,
\end{equation}
converges uniformly for $\omega$ sufficiently large, where $\tilde
f$ is given in \eqref{def:g}. Then, as $\omega \to \infty$, we have
\begin{equation}\label{eq:asy of Hadamard int}
\ddashint_0^{\infty}e^{i\omega t}\frac{f(t)}{t}dt=\left\{
                    \begin{array}{ll}
                     \frac{e^{\frac{\pi}{2}(2-\alpha)i}\omega^\alpha}{\alpha}\Gamma(1-\alpha)a_0+
\sum\limits_{\ell=1}^{m}a_{\ell}
e^{\frac{\pi}{2}(\ell-\alpha)i}
\frac{\Gamma(\ell-\alpha)}{\omega^{\ell-\alpha}}+o(\omega^{-m}), & \hbox{if $0<\alpha<1$,} \\
                      (i\frac{\pi}{2}-\gamma-\log\omega)f(0)+\sum\limits_{\ell=1}^{m}a_{\ell}
e^{\frac{\pi \ell}{2} i}
\frac{(\ell-1)!}{\omega^{\ell}}+o(\omega^{-m}), & \hbox{if
$\alpha=0$.}
                    \end{array}
                  \right.
\end{equation}
\end{theorem}
\begin{proof}
From \eqref{def:g} and \eqref{eq:asy of f}, it is readily seen that
\begin{equation}
\tilde f(t)=\frac{f(t)-a_0t^{-\alpha}}{t}
\sim\sum_{j=0}^{\infty}a_{j+1}t^{j-\alpha}
\end{equation}
as $t\to 0^+$. This, together with \cite[Thm.~1, p.~199]{wong},
leads us to the following asymptotics of Fourier integrals:
\begin{equation}\label{eq:asy of Fourier1}
\int_0^{\infty}\tilde f(t)e^{i\omega t}=
                     \sum\limits_{\ell=1}^{m-1}a_{\ell}
e^{\frac{\pi}{2}(\ell-\alpha)i}
\frac{\Gamma(\ell-\alpha)}{\omega^{\ell-\alpha}}+o(\omega^{-m}).
\end{equation}
The asymptotic expansions \eqref{eq:asy of Hadamard int} of Hadamard
finite-part integrals then follows from a combination of
\eqref{eq:Hadamard int} and \eqref{eq:asy of Fourier1}.
\end{proof}

\begin{remark}
We do not require that $f$ has an analytic continuation to the first
quadrant in Theorem \ref{thm:asy of Hadamard}. However, such
requirement is essential in the proof of Theorem \ref{thm:asy x>0}.
\end{remark}

\subsection{Some examples}

We conclude this section with applications of Theorems \ref{thm:asy
x>0} and \ref{thm:asy of Hadamard} to some concrete examples.

\begin{example}
Consider the integral
\begin{equation}\label{eq:exp example}
\dashint_{0}^{\infty}e^{i\omega t}\frac{e^{-ct}}{t-x}dt,\qquad c\geq
0.
\end{equation}
This is the one-sided oscillatory Hilbert transform of
$f(t)=e^{-ct}$. Clearly, $f$ is an entire function and satisfies all
the assumptions of our theorems. Since
\begin{equation}
f(t)\sim\sum_{j=0}^{\infty}\frac{(-c)^j}{j!}t^j
\end{equation}
as $t\to 0^+$, one has
\begin{equation}
\alpha=0, \qquad a_j=\frac{(-c)^j}{j!}, \quad j=0,1,2,\cdots,
\end{equation}
in the notation of \eqref{eq:asy of f}. Hence, from \eqref{eq:asy
x>0} and \eqref{eq:asy of Hadamard int}, it is easily seen that
\begin{equation}\label{eq:asy of exp x>0}
\dashint_{0}^{\infty}e^{i\omega t}\frac{e^{-ct}}{t-x}dt\sim i\pi
e^{i\omega x}e^{-cx}
-\sum_{\ell=0}^{\infty}\frac{\Gamma(\ell+1)}{\omega^{\ell+1}}e^{\frac{\pi}{2}(\ell+1)i}
\left(\sum_{j+k=\ell}\frac{(-c)^j}{j!x^{k+1}}\right), \quad x>0,
\end{equation}
and
\begin{equation}
\ddashint_{0}^{\infty}e^{i\omega t}\frac{e^{-ct}}{t}dt\sim
i\frac{\pi}{2}-\gamma-\log\omega+\sum\limits_{\ell=1}^{\infty}
\frac{(-c)^{\ell}}{\ell}e^{i\frac{\pi}{2}\ell}
\frac{1}{\omega^{\ell}},
\end{equation}
as $\omega\rightarrow\infty$. In particular, if $c=0$, i.e.,
$f\equiv 1$, the above two expansions can be simplified, and we have
\begin{equation}\label{asy:f=1}
\dashint_{0}^{\infty}\frac{e^{i\omega t}}{t-x}dt\sim i\pi e^{i\omega
x}-\sum_{\ell=0}^{\infty}\frac{\Gamma(\ell+1)}{(\omega
x)^{\ell+1}}e^{\frac{\pi}{2}(\ell+1)i}, \quad x>0,
\end{equation}
and
\begin{equation}
\ddashint_{0}^{\infty}\frac{e^{i\omega t}}{t}dt =
i\frac{\pi}{2}-\gamma-\log\omega.
\end{equation}
It is worthwhile to point out that the asymptotic expansion
\eqref{asy:f=1} is the same as the large $x$ expansion of
$\dashint_{0}^{\infty}\frac{e^{i\omega t}}{t-x}dt$ with $\omega>0$;
see \cite[Lem.~1]{ursell}.

\end{example}

\begin{example}\label{example:2rd}
As the second example, we consider
\begin{equation}\label{eq:example 2rd}
\dashint_{0}^{\infty}e^{i\omega t}\frac{\sqrt{t}}{(1+t)(t-x)}dt.
\end{equation}
Thus, $f(t)=\sqrt{t}/(1+t)$, which can be analytically extended to
the complex plane with a pole at $-1$ and a branch cut along the
negative axis. As $t\to 0^+$, we have
\begin{equation}
f(t)\sim\sqrt{t}\sum_{j=0}^{\infty}(-t)^j=\sum_{j=1}^{\infty}(-1)^{j+1}t^{j-\frac{1}{2}}.
\end{equation}
Hence, in the notation of \eqref{eq:asy of f}, $\alpha=1/2$,
$a_0=0$, $a_j=(-1)^{j+1}$, $j=1,2,\cdots$. An appeal to
\eqref{eq:asy x>0} and \eqref{eq:asy of Hadamard int} then gives
\begin{equation*}
\dashint_{0}^{\infty}e^{i\omega t}\frac{\sqrt{t}}{(1+t)(t-x)}dt\sim
i\pi e^{i \omega x}
\frac{\sqrt{x}}{1+x}-\sum_{\ell=1}^{\infty}\frac{\Gamma(\ell+\frac{1}{2})}{\omega^{\ell+\frac{1}{2}}}
e^{\frac{\pi}{2}(\ell+\frac{1}{2})i}\left(\sum_{j=1}^{\ell}\frac{(-1)^{j+\ell}}{x^{j}}\right)
\end{equation*}
for $x>0$, and
\begin{equation*}
\ddashint_{0}^{\infty}e^{i\omega t}\frac{\sqrt{t}}{t(1+t)}dt\sim
\sum\limits_{\ell=1}^{\infty}(-1)^{\ell+1}
e^{\frac{\pi}{2}(\ell-\frac{1}{2})i}
\frac{\Gamma(\ell-\frac{1}{2})}{\omega^{\ell-\frac{1}{2}}}.
\end{equation*}
\end{example}

\section{Computation of oscillatory Hilbert
transforms}\label{sec:computation} Although the asymptotic
expansions derived in the previous section provide essential
insights into the behaviour of the oscillatory Hilbert transforms
for large $\omega$, they are not suitable for computational purpose,
since asymptotic expansions are typically divergent, one can not
simply keep on adding terms of the expansion in order to improve the
accuracy of the approximation. In this section we shall focus on
fast numerical computation of oscillatory Hilbert transforms
\eqref{def:one-side transform}. According to the position of $x$, we
classify our discussion into three regimes, namely,
$x=\mathcal{O}(1)$ or $x\gg1$, $0<x\ll 1$ and $x=0$. Since these
regimes exhibit different asymptotic behaviour, they also require
different numerical methods.

\subsection{The regime $x=\mathcal{O}(1)$ or $x\gg1$}
If $x$ is not so close to the origin, say, $x=\mathcal{O}(1)$ or
$x\gg1$, the one-sided oscillatory Hilbert transform
\eqref{def:one-side transform} can be approximated efficiently using
the generalized Gauss-Laguerre quadrature rule. To see this, we
observe from \eqref{eq:tranfer trans} that
\begin{align}\label{eq:gauss-Laguerre}
H^{+}(f(t)e^{i\omega t})(x)&=i\pi f(x)e^{i\omega
x}+\frac{1}{\omega}\int_{0}^{\infty}e^{-q}\frac{f(\frac{iq}{\omega})}{\frac{q}{\omega}+ix}dq
\nonumber \\
&=i\pi f(x)e^{i\omega x}+\frac{\exp(-\frac{\alpha}{2}\pi
i)}{\omega^{1-\alpha}}\int_{0}^{\infty}q^{-\alpha}e^{-q}\frac{f_\alpha(\frac{iq}{\omega})}{\frac{q}{\omega}+ix}dq,
\end{align}
with $f_\alpha(t)=t^{\alpha}f(t)$ and where $\alpha$ is defined as
in \eqref{eq:asy of f}. We have made use of the change of variable
$q=\omega p$ in the first equality. On account of \eqref{eq:asy of
f}, it is easily seen that
$f_\alpha(\frac{iq}{\omega})/(\frac{q}{\omega}+ix)$
behaves like a polynomial near the origin, if $x=\mathcal{O}(1)$ or
$x\gg1$. This is exactly the situation that can be handled by the
generalized Gauss-Laguerre quadrature rule. Let
$\{t_j,w_j\}_{j=1}^{n}$ be the nodes and weights of the generalized
Gauss-Laguerre quadrature rule with respect to the weight
$t^{-\alpha}e^{-t}$, with $0\leq\alpha<1$. Then the one-sided
oscillatory Hilbert transform $H^{+}(f(t)e^{i\omega t})(x)$ is
approximated by
\begin{align}\label{eq:Laguerre quadrature rule}
Q_n(f,\omega,x)=i\pi f(x)e^{i\omega
x}+\frac{\exp(-\frac{\alpha}{2}\pi
i)}{\omega^{1-\alpha}}\sum_{k=1}^{n}w_k\frac{f_\alpha(\frac{it_k}{\omega})}{\frac{t_k}{\omega}+ix}.
\end{align}
Applying the error expression of the $n$-point generalized
Gauss-Laguerre quadrature rule \cite[p. 223]{Davis}, we can
estimate, for each fixed $x$, the quadrature error as follows:
\begin{align}
H^{+}(f(t)e^{i\omega
t})(x)-Q_n(f,\omega,x)&=\frac{\exp(-\frac{\alpha}{2}\pi
i)}{\omega^{1-\alpha}}\left(\int_{0}^{\infty}q^{-\alpha}e^{-q}\frac{f_\alpha(i\frac{q}{\omega})}{\frac{q}{\omega}+ix}dq
-\sum_{k=1}^{n}w_k\frac{f_\alpha(\frac{it_k}{\omega})}{\frac{t_k}{\omega}+ix}
\right)  \nonumber \\
&=\frac{\exp(-\frac{\alpha}{2}\pi
i)}{\omega^{1-\alpha}}\frac{n!\Gamma(n-\alpha+1)}{(2n)!}\left(\frac{f_\alpha(i\frac{q}{\omega})}{\frac{q}{\omega}+ix}\right)^{(2n)}\bigg|_{q=\xi}
\nonumber \\
&=\mathcal{O}(\omega^{-2n-1+\alpha})
\end{align}
for some constant $\xi\in\mathbb{C}$ as $\omega\rightarrow\infty$.
Note that one gains the factor $\omega^{-2n}$ from taking the
$(2n)$-th order derivative of a function of $q$ that depends only on
$q/\omega$. We observe that the accuracy of the quadrature rule
(3.2) rapidly improves as $\omega$ grows. This result may come as a
surprise here, since we started out with a highly oscillatory
integral that typically requires a growing number of quadrature
points to evaluate as the frequency increases. Yet, we like to point
out that this result is entirely parallel to the case of finite
Fourier integrals of the form \eqref{eq:Fourier integrals}, where an
application of Gauss-Laguerre yields similar behaviour for large
$\omega$. In what follows, we give several examples to illustrate
the convergence of our quadrature rule $Q_n(f,\omega,x)$. Throughout
this paper, all the computations have been performed using Maple 14
with 32-digit arithmetic\footnote{The use of increased precision is
just to show the convergence rates of our methods.}.

\begin{example}
Let us consider \eqref{eq:exp example} with $c=1$, that is,
$f(t)=e^{-t}$, and one has $\alpha=0$ in \eqref{eq:Laguerre
quadrature rule}. The exact solution of \eqref{eq:exp example} can
be obtained by using the definition of the Cauchy principal value
integral and taking the series expansion of the exponential
function. The result is
\begin{align}\label{eq: exact solution}
&\dashint_{0}^{\infty}e^{i\omega t}\frac{e^{-ct}}{t-x}dt \nonumber
\\&=
e^{(-c+i\omega)x}\bigg[\mathrm{Ei}(1,(c-i\omega)x)-2\sum_{\ell=0}^{\infty}
\frac{c^{2\ell+1}}{(2\ell+1)!}\left(\sum_{k=0}^{2\ell}k!\binom{2\ell}{k}\frac{x^{2\ell-k}}{\omega^{k+1}}\sin(\omega
x+\frac{k\pi}{2})\right) \nonumber \\
&~~~~+i\left\{2\mathrm{Si}(\omega x)-2\sum_{\ell=0}^{\infty}
\frac{c^{2\ell+2}}{(2\ell+2)!}\left(\sum_{k=0}^{2\ell+1}
k!\binom{2\ell+1}{k}\frac{x^{2\ell+1-k}}{\omega^{k+1}}\cos(\omega
x+\frac{k\pi}{2})\right)\right\}\bigg],
\end{align}
where
\begin{equation}
\mathrm{Si}(x)=\int_0^{x}\frac{\sin t}{t}dt,\qquad x \geq 0,
\end{equation}
is the sine integral; cf. \cite{Abram}. In the case of $c=0$, one
can check that
\begin{align}\label{eq:c=0}
\dashint_{0}^{\infty}\frac{e^{i\omega t}}{t-x}dt &=e^{i\omega
x}\left(\mathrm{Ei}(1,-i\omega x)+i 2 \mathrm{Si}(\omega x) \right)
=e^{i\omega x} \left(i \pi + \mathrm{Ei}(1,i\omega x)\right).
\end{align}
We compute the error for $x=1$ and $x=5$ with different frequency
$\omega$ and $n$ ranging from $2$ to $10$. The results are
illustrated in Figure \ref{fig:e^t}.
\begin{figure}[h]
\centering
\begin{overpic}
[width=5.5cm]{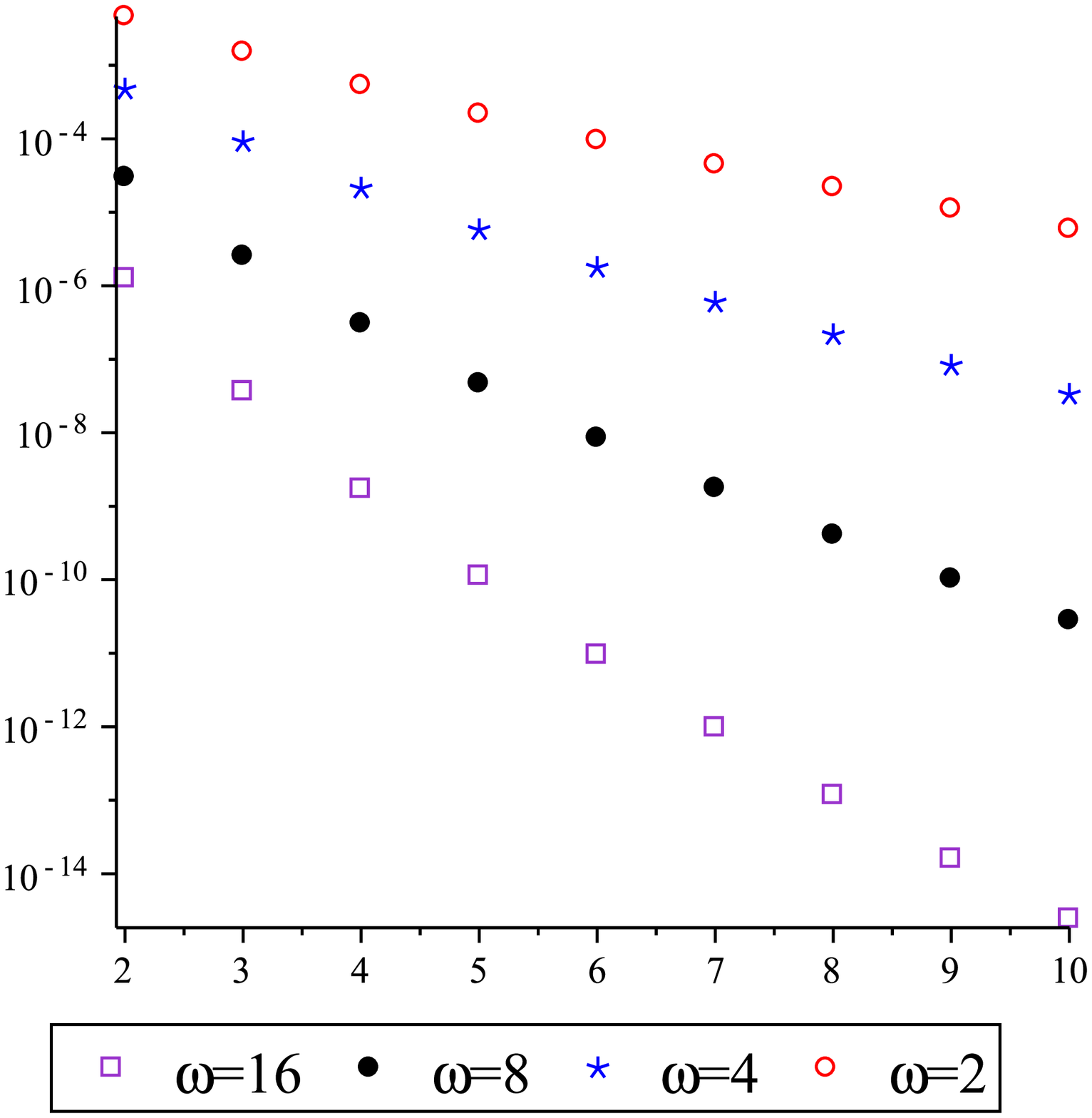} 
\end{overpic}
\qquad
\begin{overpic}
[width=5.5cm]{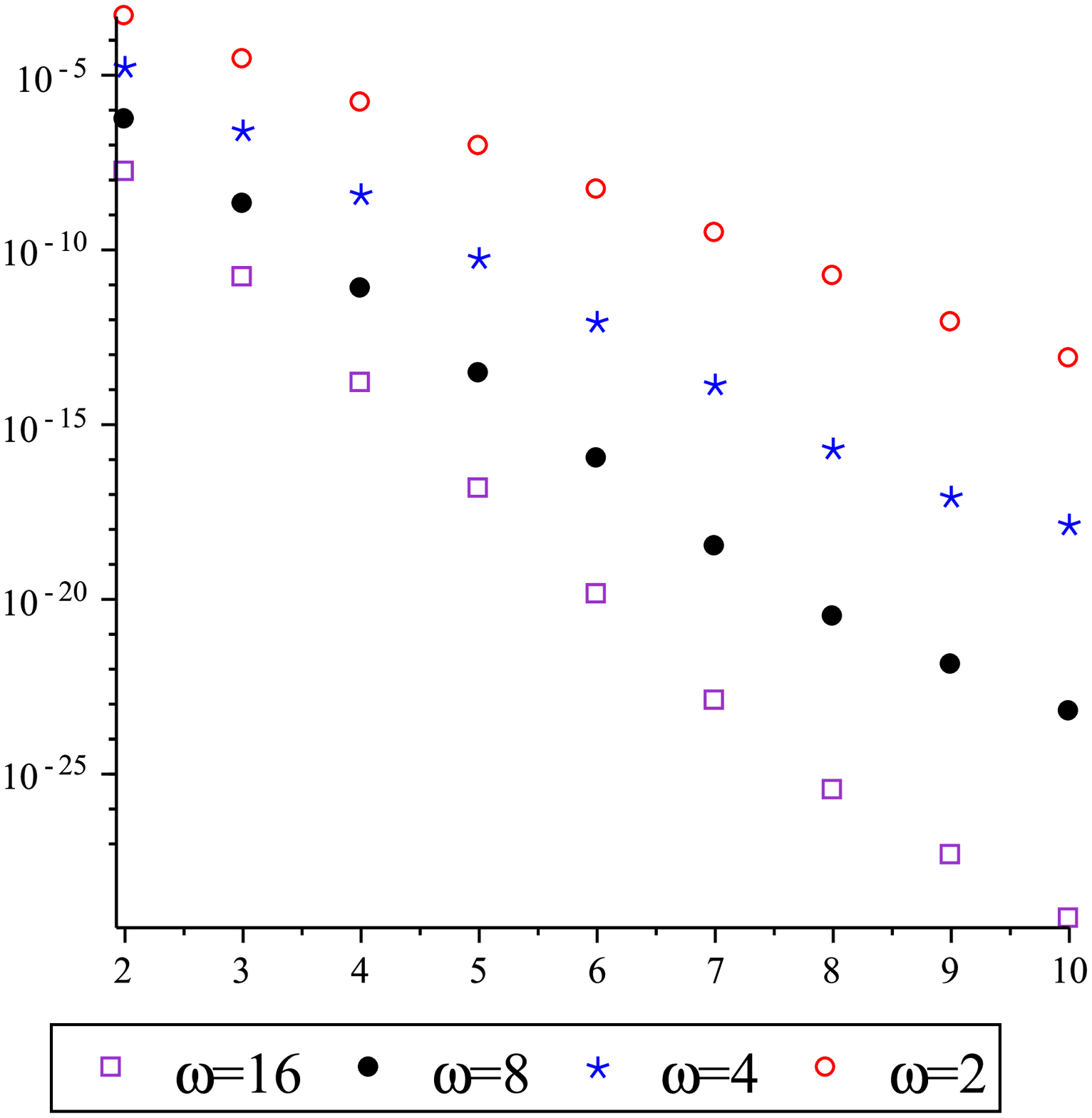} 
\end{overpic}
\caption{The error of the quadrature $Q_n(f,\omega,x)$ for $x=1$
(left) and $x=5$ (right) with $f(t)=e^{-t}$ and $n$ ranging from $2$
to $10$.} \label{fig:e^t}
\end{figure}
\end{example}

\begin{example}
We next consider the function $f(t)=\frac{\cos t}{\sqrt[3]{t}}$,
which grows exponentially in the complex plane. We can check that it
satisfies \eqref{eq:growth condition} with $d=1$, and
$\alpha=\frac{1}{3}$ in \eqref{eq:Laguerre quadrature rule}. The
error is illustrated in Figure \ref{fig:double pole R1}.
\begin{figure}[h]
\centering
\begin{overpic}
[width=5.5cm]{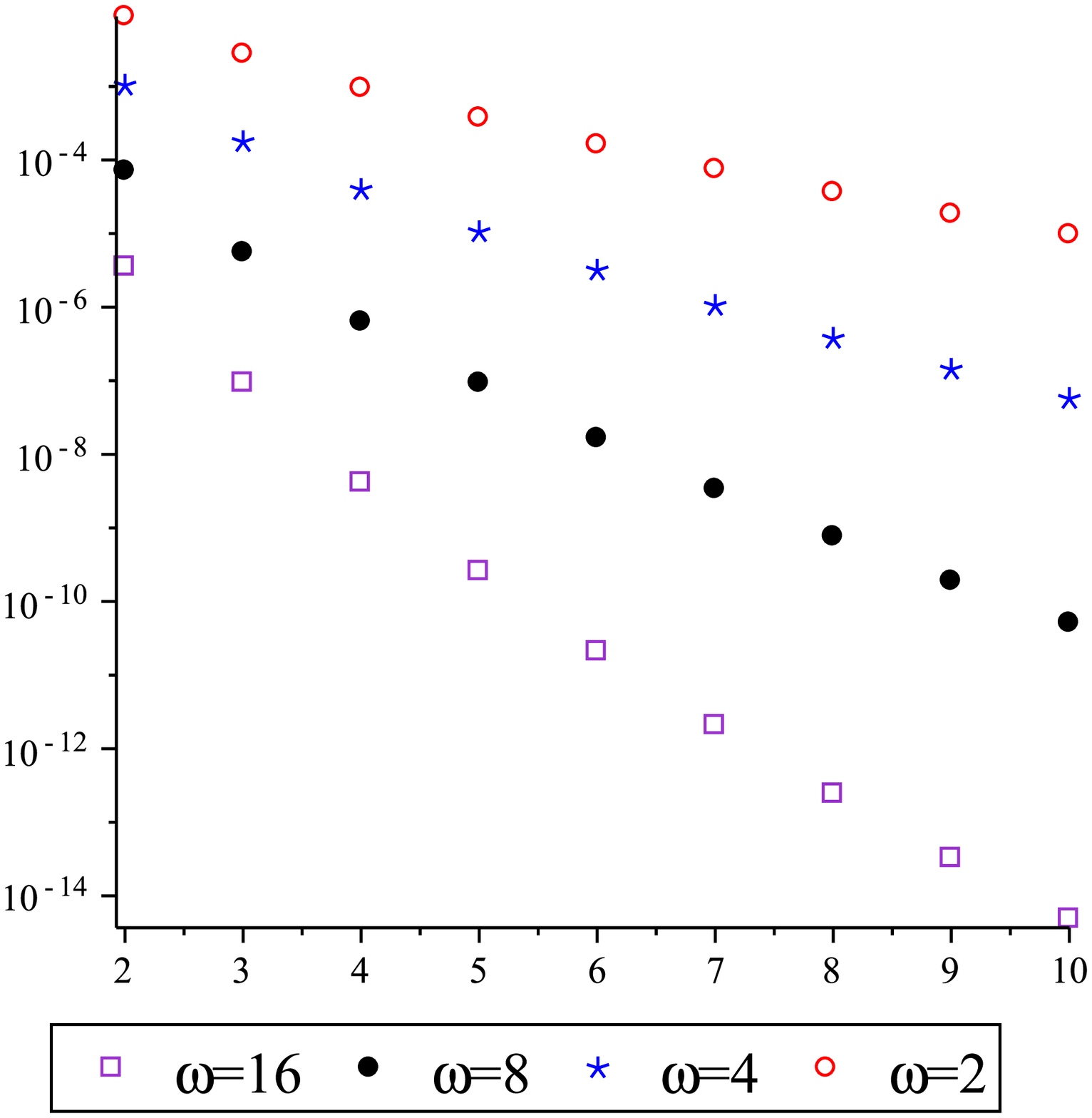}
\end{overpic}
\qquad
\begin{overpic}
[width=5.5cm]{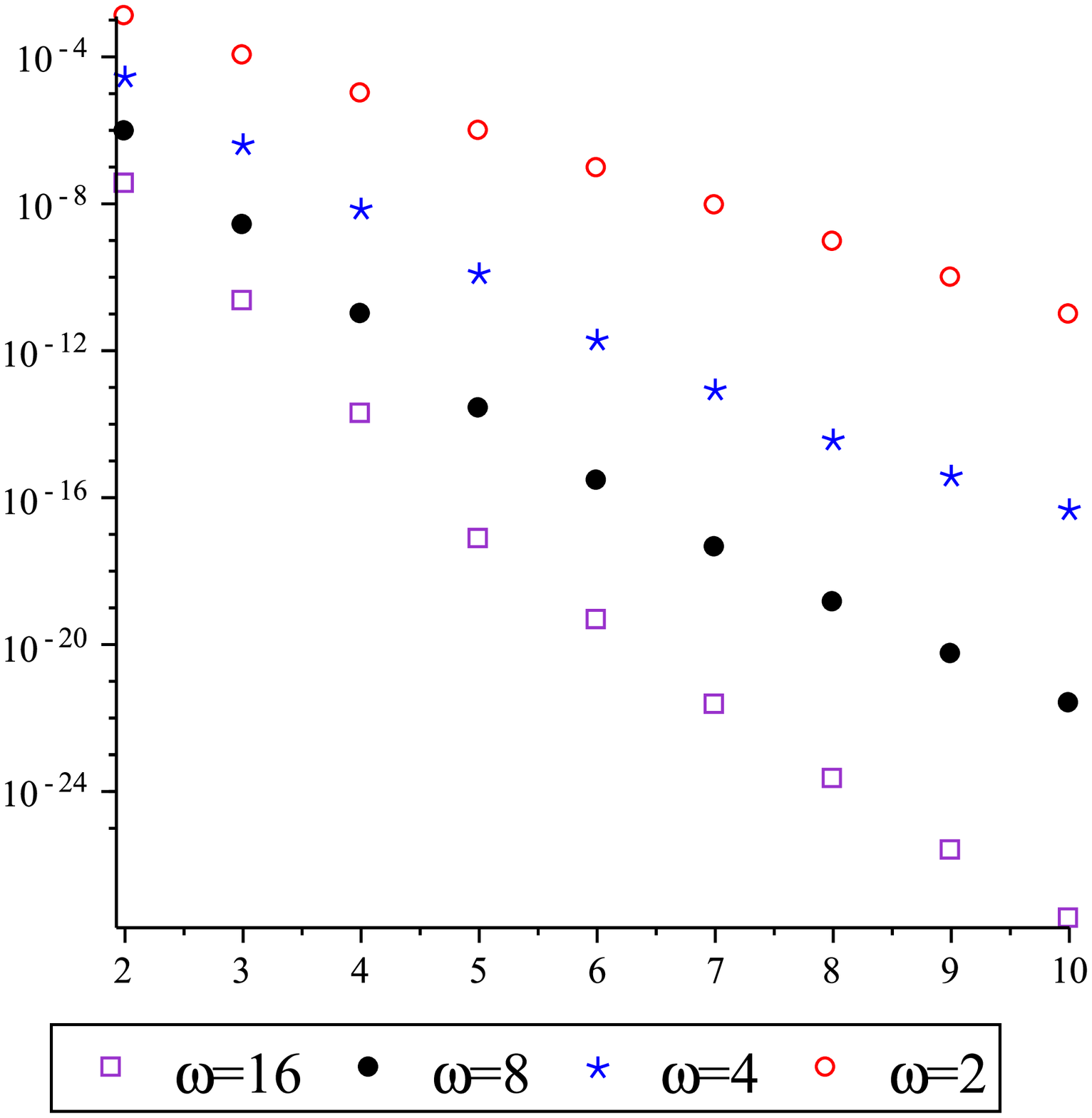}
\end{overpic}
 \caption{The error of the quadrature $Q_n(f,\omega,x)$ for $x=1$ (left) and $x=5$ (right) with  $f(t)=\frac{\cos t}{\sqrt[3]{t}}$ and $n$ ranging
from $2$ to $10$.} \label{fig:double pole R1}
\end{figure}
\end{example}

\begin{example}
We return to Example \ref{example:2rd}, i.e., $f(t)=\sqrt{t}/(1+t)$,
which corresponds to $\alpha=1/2$. Numerical results are displayed
in Figure \ref{fig:branch case R1}.
\begin{figure}[h]
\centering
\begin{overpic}
[width=5.5cm]{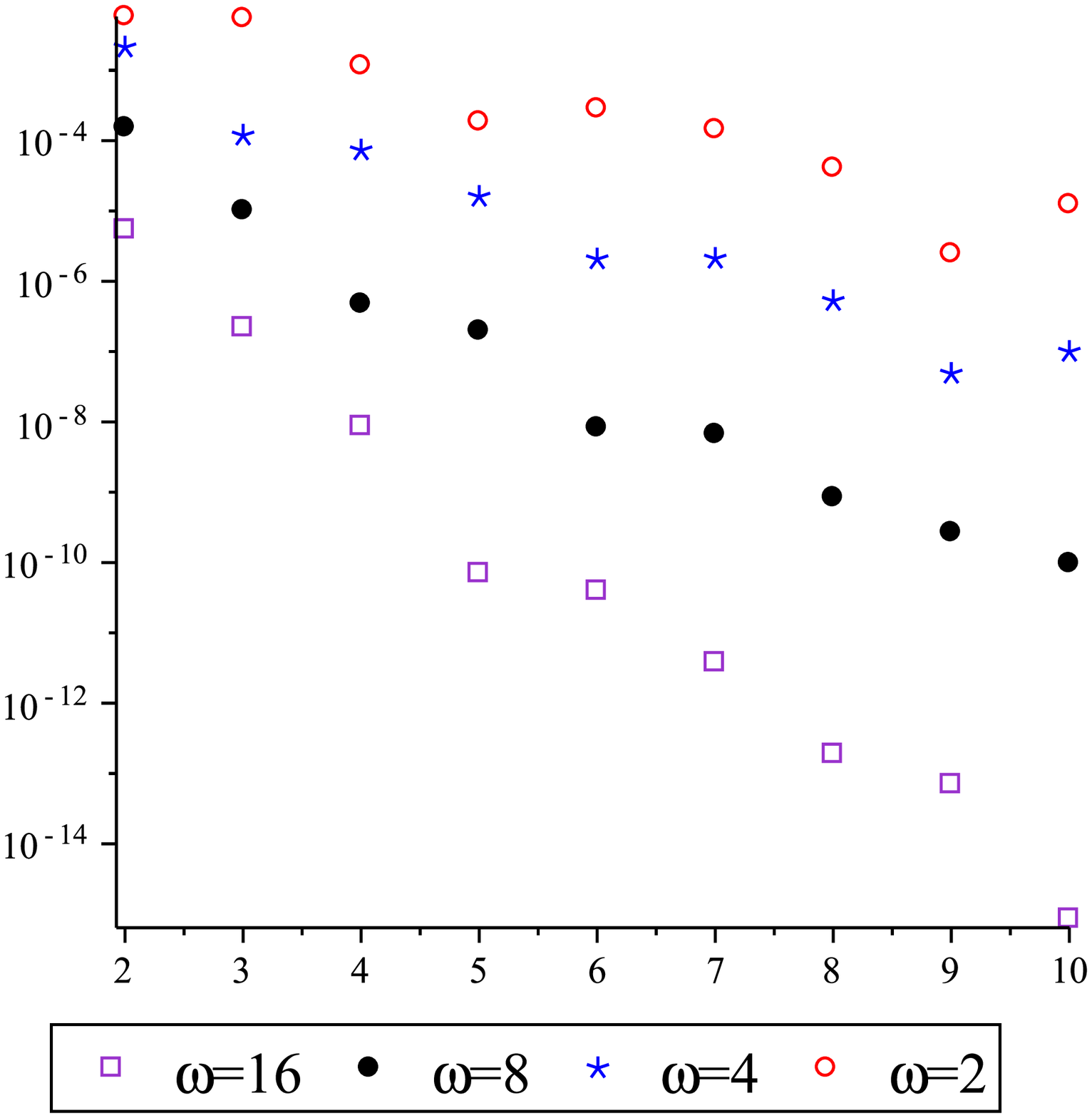}
\end{overpic}
\qquad
\begin{overpic}
[width=5.5cm]{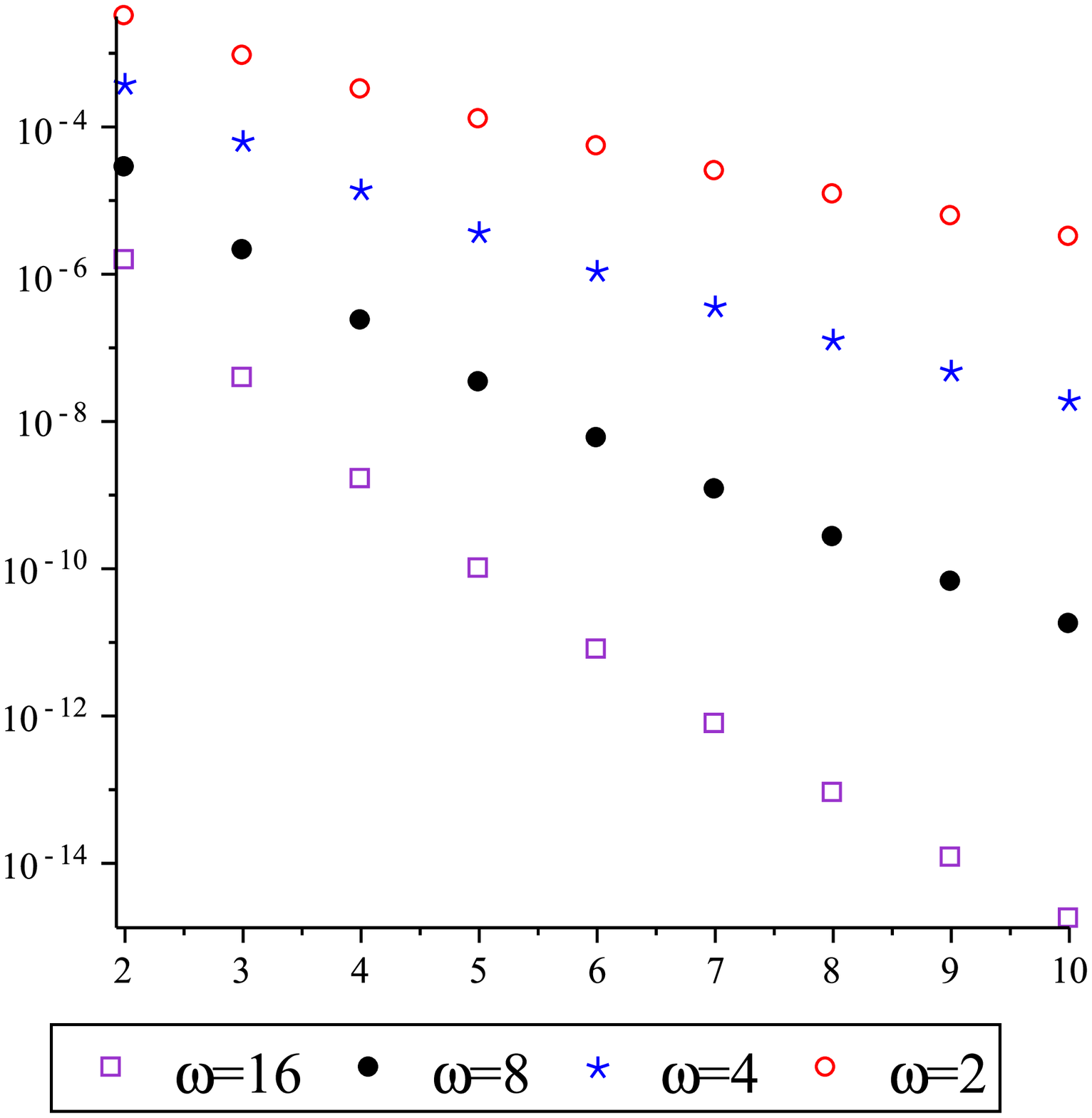}
\end{overpic}
 \caption{The error of the quadrature $Q_n(f,\omega,x)$ for $x=1$ (left) and $x=5$ (right) with $f(t)=\frac{\sqrt{t}}{1+t}$ and $n$ ranging
from $2$ to $10$.} \label{fig:branch case R1}
\end{figure}
\end{example}

All these examples show that the accuracy of quadrature rule
\eqref{eq:Laguerre quadrature rule} improves rapidly as $n$
increases. Meanwhile, the convergence is faster for larger $\omega$.

\subsection{The regime $0<x\ll1$}
When $x$ is close the origin, i.e., $0<x\ll1$, the accuracy of the
generalized Gauss-Laguerre rule deteriorates since the integrand on
the right hand side of \eqref{eq:gauss-Laguerre} is nearly singular.
To this end, we make the following decomposition of oscillatory
Hilbert transforms:
\begin{align}\label{eq:decomposition R2}
H^{+}(f(t)e^{i\omega t})(x)
=\dashint_{0}^{a}\frac{f(t)}{t-x}e^{i\omega
t}dt+\int_{a}^{\infty}\frac{f(t)}{t-x}e^{i\omega t}dt,
\end{align}
where $a$ is a positive number larger than $x$. Let the two
integrals on the right hand side of \eqref{eq:decomposition R2} be
denoted by $I_1(x)$ and $I_2(x)$, respectively. Note that the
integral $I_2(x)$ is no longer singular. A fast computation of
oscillatory Hilbert transforms in this regime is then reduced to the
numerical study of $I_1(x)$ and $I_2(x)$, which will be discussed in
the next two sections.

\subsubsection{Computation of $I_1(x)$}
The integral $I_1(x)$ is a finite oscillatory Hilbert transform. By
scaling the interval of integration to $(-1,1)$, we have
\begin{align}\label{eq:scaling}
I_1(x)&=\dashint_{0}^{a}\frac{f(t)}{t-x}e^{i\omega t}dt\nonumber \\
&=e^{i\tilde
\omega}\dashint_{-1}^{1}\frac{f(\frac{a}{2}(y+1))}{y-\tau}e^{i\tilde\omega
y}dy \nonumber \\
&=\left(\frac{a}{2}\right)^{-\alpha}e^{i\tilde{\omega}}\dashint_{-1}^{1}\frac{(y+1)^{-\alpha}h(y)}{y-\tau}e^{i\tilde{\omega}y}dy,
\end{align}
where $\tilde{\omega}=\frac{\omega a}{2}$,
$\tau=\frac{2x}{a}-1\in(-1,1)$, and
\begin{equation}
h(y)=f_{\alpha}(\frac{a}{2}(y+1))=\left(\frac{a}{2}(y+1)\right)^{\alpha}f\left(\frac{a}{2}(y+1)\right).
\end{equation}


The computation of \eqref{eq:scaling} with $\alpha=0$ has been
discussed in \cite{wang1}. Let $p_N(y)$ be the Lagrange polynomial
which interpolates $h(y)$ at the Clenshaw-Curtis points
$y_j=\cos\left(\frac{j\pi}{N}\right)$, $j=0,\ldots,N$. Then one has
(see \cite{Clenshaw})
\begin{align}\label{eq: ChebyOp}
p_N(y)=\sum_{k=0}^{N}{''}a_k^NT_k(y)
\end{align}
with
\begin{align}\label{def:akN}
a_k^N=\frac{2}{N}\sum_{j=0}^{N}{''}h(y_j)T_j(y_k), \quad
k=0,\ldots,N,
\end{align}
where the double prime denotes a sum whose first and last terms are
halved and $T_j(y)$ is the Chebyshev polynomial of the first kind of
degree $j$. The coefficients $a_k^N$ can be computed efficiently by
FFT \cite{gentleman}. Replacing $h(y)$ in \eqref{eq:scaling} by
$p_N(y)$, we have
\begin{align}\label{eq:previous methods}
\dashint_{-1}^{1}\frac{f(y)}{y-\tau}e^{i\tilde{\omega} y}dy & \simeq
\dashint_{-1}^{1}\frac{p_N(y)}{y-\tau}e^{i\tilde{\omega} y}dy \nonumber\\
&=\dashint_{-1}^{1}\frac{p_N(y)-p_N(\tau)}{y-\tau}e^{i\tilde{\omega}
y}dy + p_N(\tau)\dashint_{-1}^{1}\frac{e^{i\tilde{\omega}
y}}{y-\tau}dy \nonumber\\
&=
2\sum_{k=0}^{N}{''}a_k^N\left[\sum_{n=0}^{k-1}{'}T_n(\tau)M_{k-1-n}\right]
+ p_N(\tau)\dashint_{-1}^{1}\frac{e^{i\tilde{\omega} y}}{y-\tau}dy,
\end{align}
where
\begin{align*}
M_n=\int_{-1}^{1}U_n(y)e^{i\tilde{\omega }y}dy, \quad n\geq0,
\end{align*}
and where $U_n(y)$ is the Chebyshev polynomial of the second kind of
degree $n$. The last integral in \eqref{eq:previous methods} can be
computed in closed form and the $M_n$ can be computed by using a
three-term recurrence relation. The advantage of this method is that
it converges uniformly with respect to $\tau\in(-1,1)$ as
$N\rightarrow\infty$, provided $h(y)$ is analytic in a small
neighborhood containing $[-1,1]$. However, this method costs
$\mathcal{O}(N^2)$ operations which make it inefficient for large
$N$.

In the following we shall present a more efficient implementation,
which costs only $\mathcal{O}(N\log_2N)$ operations, to compute the
finite oscillatory Hilbert transform \eqref{eq:scaling}. The key
observation is that we can write $(p_N(y)-p_N(\tau))/(y-\tau)$ in
terms of $T_k(y)$ and this process can be performed in only
$\mathcal{O}(N)$ operations.

We approximate $I_1(x)$ by
\begin{align}\label{def: finiteoscill}
Q_{N}^{(1)}(f,\tilde{\omega},x)&=\left(\frac{a}{2}\right)^{-\alpha}e^{i\tilde
\omega}\left(\dashint_{-1}^{1}(y+1)^{-\alpha}\frac{p_N(y)}{y-\tau}e^{i\tilde{\omega}y}dy\right)
\nonumber \\
&=\left(\frac{a}{2}\right)^{-\alpha}e^{i\tilde
\omega}\bigg(\int_{-1}^{1}(y+1)^{-\alpha}\frac{p_N(y)-p_N(\tau)}{y-\tau}e^{i\tilde{\omega}y}dy \nonumber \\
&~~~~~~~~~~~~~~~~~~~~~+p_N(\tau)\dashint_{-1}^{1}(y+1)^{-\alpha}\frac{e^{i\tilde{\omega}y}}{y-\tau}dy
\bigg).
\end{align}
We next expand $(p_N(y)-p_N(\tau))/(y-\tau)$ in terms of $T_k(y)$
and obtain
\begin{align}\label{eq:expand Pny-pntau}
\frac{p_N(y)-p_N(\tau)}{y-\tau}=\sum_{k=0}^{N-1}{'}b_k^NT_k(y),
\end{align}
where the prime denotes the summation whose first term is halved.
The coefficients $b_k^N$ satisfy a three-term recurrence relation
\begin{align}
b_{k-1}^N=2a_k^N+2\tau b_k^N-b_{k+1}^N, \quad k=N-1,\ldots,1,
\end{align}
and the first two initial values are given by $b_N^N=0$ and
$b_{N-1}^{N}=a_N^N$; see \cite{hasegawa}. Inserting \eqref{eq:expand
Pny-pntau} into \eqref{def: finiteoscill} gives us
\begin{align}\label{def:modified CC methods}
Q_N^{(1)}(f,\tilde{\omega},x)
&=\left(\frac{a}{2}\right)^{-\alpha}e^{i\tilde
\omega}\left(\sum_{k=0}^{N-1}{'}b_k^NZ_k^{(\alpha)}+p_N(\tau)\dashint_{-1}^{1}(y+1)^{-\alpha}\frac{e^{i\tilde{\omega}y}}{y-\tau}dy\right),
\end{align}
where
\begin{align}\label{def: Cheb moment1}
Z_k^{(\alpha)}:=Z_k(\alpha,\tilde{\omega})=\int_{-1}^{1}(y+1)^{-\alpha}T_k(y)e^{i\tilde{\omega}y}dy,
\quad k\geq0.
\end{align}
Now, we only need to compute the moments $Z_k^{(\alpha)}$ and the
integral on right hand side of \eqref{def:modified CC methods}
efficiently. It turns out that $Z_k^{(\alpha)}$ satisfy the
following four-term recurrence relation (see \cite{piessens}):
\begin{align}\label{eq: recurrence 4}
&i\tilde \omega(n-1)
Z_{n+1}^{(\alpha)}+[2(n-\alpha+1)(n-1)+i\tilde{\omega}(n-2)]Z_{n}^{(\alpha)}\nonumber
\\
&~~~+[2n(n+\alpha-2)-i\tilde{\omega}(n+1)]Z_{n-1}^{(\alpha)}-i\tilde{\omega}
nZ_{n-2}^{(\alpha)}=-2^{2-\alpha}e^{i\tilde{\omega}}
\end{align}
with the first three initial values given by
\begin{align}
Z_{0}^{(\alpha)}&=\frac{e^{-i\tilde{\omega}}}{\tilde{\omega}^{1-\alpha}}e^{\frac{1}{2}\pi
i(1-\alpha)}\gamma(1-\alpha,-2i\tilde{\omega}),\nonumber \\
Z_{1}^{(\alpha)}&=\frac{e^{-i\tilde{\omega}}}{\tilde{\omega}^{2-\alpha}}e^{\frac{1}{2}\pi
i(2-\alpha)}\gamma(2-\alpha,-2i\tilde{\omega})-Z_{0}^{(\alpha)},\nonumber \\
Z_{2}^{(\alpha)}&=\frac{2e^{-i\tilde{\omega}}}{\tilde{\omega}^{3-\alpha}}e^{\frac{1}{2}\pi
i(3-\alpha)}\gamma(3-\alpha,-2i\tilde{\omega})-4Z_{1}^{(\alpha)}-3Z_{0}^{(\alpha)},
\nonumber
\end{align}
where
\begin{align}\label{incomplete gamma2}
\gamma(a,z)=\int_{0}^{z}t^{a-1}e^{-t}dt, \qquad \Re(a)>0,
\end{align}
is the incomplete gamma function \cite[p. 260]{Abram}. The
recurrence relation \eqref{eq: recurrence 4} can be used to
calculate $Z_k^{(\alpha)}$ stably in the forward direction provided
$N\leq2\tilde{\omega}$. If $N>2\tilde{\omega}$, we can compute the
additional moments $Z_k^{(\alpha)}$, $2\tilde{\omega}< k\leq N$
stably by solving a boundary value problem with two starting values
and one ending value \cite{piessens}. To see this, let
$n_0:=[2\tilde{\omega}]$, where $[~]$ denotes the integer part, and
choose a positive integer $N_1\geq \max\{n_0,N\}$. We then define a
matrix $A=(a_{j,k})_{j,k=1}^{N_1-n_0+1}$ of size $N_1-n_0+1$ by
\begin{align}\label{def: matrix}
a_{j,k}=\left\{
               \begin{array}{ll}
                    i\tilde{\omega}(n_0+j-2), & \hbox{$k=j+1$,} \\
                    2(n_0-\alpha+j)(n_0+j-2)+i\tilde\omega(n_0+j-3), & \hbox{$k=j$,}\\
                    2(n_0+j-1)(n_0+\alpha+j-3)-i\tilde\omega(n_0+j), &\hbox{$k=j-1$,}\\
                    -i\tilde{\omega}(n_0+j-1), & \hbox{$k=j-2$,} \\
                    \end{array}
               \right.
\end{align}
and a vector $\vec{b}=(b_1,b_2,\ldots,b_{N_1-n_0+1})^{T}$ by
\begin{align}\label{def: right vector}
b_j=\left\{
\begin{array}{ll}
                    -2^{2-\alpha}e^{i\tilde{\omega}}-(2n_0(n_0+\alpha-2)-i\tilde{\omega}(n_0+1))Z_{n_0-1}^{(\alpha)}+i\tilde{\omega} n_0Z_{n_0-2}^{(\alpha)}, & \hbox{$j=1$,} \\
                    -2^{2-\alpha}e^{i\tilde{\omega}}+i\tilde{\omega}(n_0+1)Z_{n_0-1}^{(\alpha)}, & \hbox{$j=2$,}\\
                    -2^{2-\alpha}e^{i\tilde{\omega}}, &\hbox{$j=3,\ldots, N_1-n_0$,}\\
                    -2^{2-\alpha}e^{i\tilde{\omega}}-i\tilde{\omega}(N_1-1)Z_{N_1+1}^{(\alpha)}, & \hbox{$j=N_1-n_0+1$.} \\
                    \end{array}
\right.
\end{align}
where the superscript $^T$ stands for transpose. Here, the moments
$Z_{n_0-2}^{(\alpha)}$ and $Z_{n_0-1}^{(\alpha)}$ can be obtained by
using the forward recurrence \eqref{eq: recurrence 4}, while the
value of $Z_{N_1+1}^{(\alpha)}$ in the last component can be set
equal to zero if the selected parameter $N_1$ is sufficiently large.
We have that the vector consisting of the additional moments defined
by
\begin{align*}
\vec{Z}=(Z_{n_0}^{(\alpha)},Z_{n_0+1}^{(\alpha)},\ldots,Z_{N_1}^{(\alpha)})^{T},
\end{align*}
is the unique solution of the linear system
\begin{align}\label{eq: linear system}
A\vec{x}=\vec{b}.
\end{align}

To compute the integral on right hand side of \eqref{def:modified CC
methods}, we note that
\begin{align}\label{eq: closed form}
\dashint_{-1}^{1}(y+1)^{-\alpha}\frac{e^{i\tilde{\omega}y}}{y-\tau}dy&=
\frac{e^{i\tilde\omega\tau}}{(1+\tau)^{\alpha}}(i\pi+e^{-i\alpha\pi}\Gamma(1-\alpha)\Gamma(\alpha,i\tilde\omega(1+\tau)))\nonumber\\
&~~~~~-\frac{ie^{i\tilde\omega}}{\tilde\omega}\int_{0}^{\infty}e^{-t}\frac{1}{(2+\frac{it}{\tilde\omega})^{\alpha}(1-\tau+\frac{it}{\tilde\omega})}dt,
\end{align}
where the last integral is free from singularity as $\tau\rightarrow
-1$, and can therefore be conveniently evaluated by Gauss-Laguerre
quadrature rule. For the case $\alpha=0$, it can be further written
in a closed form (see \cite{capobianco}):
\begin{align}\label{eq:explicit form}
\dashint_{-1}^{1}\frac{e^{i\tilde{\omega}y}}{y-\tau}dy&=\cos(\tilde{\omega}\tau)
\left[\mathrm{Ci}(u_1)-\mathrm{Ci}(u_2)\right]
-\sin(\tilde{\omega}\tau)[\mathrm{Si}(u_1)+\mathrm{Si}(u_2)]     \nonumber \\
&~~~
+i\left[\sin(\tilde{\omega}\tau)[\mathrm{Ci}(u_1)-\mathrm{Ci}(u_2)]
+\cos(\tilde{\omega}\tau)[\mathrm{Si}(u_1)+\mathrm{Si}(u_2)]\right],
\end{align}
where
\begin{equation}
\mathrm{Ci}(x)=-\int_x^{\infty}\frac{\cos t}{t}dt,\qquad x>0,
\end{equation}
is the cosine integral, $u_1=\tilde{\omega}(1-\tau)$ and
$u_2=\tilde{\omega}(1+\tau)$.

Based on the above discussion, we outline our algorithm with
estimation of computational complexity as follows:
\newcounter{alg}
\begin{list}{\arabic{alg}.}
{\usecounter{alg}\setlength{\parsep}{1ex}\setlength{\itemsep}{1ex}}
\item[] {\bf Algorithm I:} Computation of $I_1(x)$
\item   Choose a positive constant $a$ larger than $x$ and $N_1>\max\{n_0,N\}$.
\item   Compute the coefficients $\{a_k^N\}_{k=0}^{N}$ in \eqref{def:akN} by using FFT
with $\mathcal{O}(N\log_2 N)$ operations.
\item  If $N\leq2\tilde\omega$, evaluate the moments $\{Z_k^{(\alpha)}\}_{k=0}^{N-1}$ from the recurrence relation
\eqref{eq: recurrence 4} in $\mathcal{O}(N)$ operations. If
$N>2\tilde\omega$, we compute the additional moments
$\{Z_k^{(\alpha)}\}_{k=n_0}^{N-1}$ by solving \eqref{eq: linear
system}, in $\mathcal{O}(N_1)$ operations\footnote{This can be done with a minor adaptation of Oliver's method for the LU decomposition of a tridiagonal matrix \cite{oliver}. We omit the details.}.  
\item  Compute $p_N(\tau)$ from its barycentric form in
$\mathcal{O}(N)$ operations \cite{Berrut}.
\item  Calculate $Q_{N}^{(1)}(f,\tilde{\omega},x)$ in $\mathcal{O}(N)$
operations by the Clenshaw algorithm:
\begin{align*}
& S_{-1}=0, S_0=\frac{1}{2}Z_0^{(\alpha)}, W_0=0,\\
&  \begin{cases}
     W_k=W_{k-1}+2a_k^NS_{k-1}, \quad k=1,2,\ldots,N-1,\\
     S_k=Z_k^{(\alpha)}+2\tau S_{k-1}-S_{k-2},\\
  \end{cases}\\
&
Q_{N}^{(1)}(f,\tilde{\omega},x)=\left(\frac{a}{2}\right)^{-\alpha}e^{i\tilde{\omega}}\left(W_{N-1}+S_{N-1}a_N^N+p_N(\tau)\dashint_{-1}^{1}(y+1)^{-\alpha}\frac{e^{i\tilde{\omega}y}}{y-\tau}dy\right).
\end{align*}
The integral
$\dashint_{-1}^{1}(y+1)^{-\alpha}\frac{e^{i\tilde{\omega}y}}{y-\tau}dy$
can be evaluated efficiently from \eqref{eq: closed form} by using a
Gauss-Laguerre quadrature rule.
\end{list}

The total computational complexity of computing $I_1(x)$ is
$\mathcal{O}(N\log_2N)$ operations if $N\leq2\tilde\omega$ and is
$\mathcal{O}(N\log_2N) + \mathcal{O}(N_1)$ operations if
$N>2\tilde\omega$. Therefore, even when we choose
$N_1=\mathcal{O}(N\log_2N)$, the total computational complexity of
computing $I_1(x)$ is still $\mathcal{O}(N\log_2 N)$. Meanwhile,
this cost is independent of $\omega$.

In the case $\alpha=0$, i.e., $f$ is analytic in the neighborhood of
the origin, the above algorithm can be further simplified. Indeed,
if $\alpha=0$, we have
\begin{align}\label{eq:Zk and Mk}
Z_k^{(0)}=\frac{1}{i\tilde{\omega}}\left(e^{i\tilde{\omega}}-(-1)^ke^{-i\tilde{\omega}}\right)-\frac{k}{i\tilde{\omega}}M_{k-1},
\quad  k\geq1.
\end{align}
Since $M_k$ satisfies a three-term recurrence relation (see
\cite{dominguez,wang1})
\begin{align}\label{eq: recuurence 2}
M_l+\frac{2l}{i\tilde{\omega}}M_{l-1}-M_{l-2}=
\frac{2}{i\tilde{\omega}}\left(e^{i\tilde{\omega}}-(-1)^le^{-i\tilde{\omega}}\right),
\quad l\geq 2,
\end{align}
with initial values $M_0=\frac{2\sin\tilde\omega}{\tilde\omega}$,
$M_1=4i(\frac{\sin\tilde\omega}{\tilde\omega^2}-\frac{\cos\tilde\omega}{\tilde\omega})$,
the four-term recurrence relation \eqref{eq: recurrence 4} can
readily be reduced to a three-term recurrence relation. In practice,
one can evaluate $M_k$ from \eqref{eq: recuurence 2} for $N\leq
\tilde{\omega}$. If $N>\tilde{\omega}$, the additional moments
$M_k$, $\tilde{\omega}<k\leq N$ can be evaluated stably by solving a
tridiagonal system \cite{dominguez}. We then have the following
simpler algorithm:
\begin{list}{\arabic{alg}.}
{\usecounter{alg}\setlength{\parsep}{1ex}\setlength{\itemsep}{1ex}}
\item[] {\bf Algorithm II:} Computation of $I_1(x)$ with $\alpha=0$
\item   Choose a positive constant $a$ larger than $x$.
\item   Compute the coefficients $\{a_k^N\}_{k=0}^{N}$ by using FFT in
$\mathcal{O}(N\log_2 N)$ operations.
\item  Evaluate the moments $\{M_k\}_{k=0}^{N-1}$ by \eqref{eq: recuurence 2} if $N\leq\tilde\omega$. If $N>\tilde\omega$, we compute
the additional moments $\{M_k\}_{k=[\widetilde{\omega}]}^{N-1}$ by
the second phase algorithm in \cite{dominguez}. Then we evaluate
$\{Z_k^{(0)}\}_{k=0}^{N-1}$ from \eqref{eq:Zk and Mk}, in
$\mathcal{O}(N)$ operations.  
\item  Compute $p_N(\tau)$ from its barycentric form in
$\mathcal{O}(N)$ operations \cite{Berrut}.
\item  Calculate $Q_{N}^{(1)}(f,\tilde{\omega},x)$ in $\mathcal{O}(N)$ operations
by the Clenshaw algorithm:
\begin{align*}
& S_{-1}=0, S_0=\frac{1}{2}Z_0^{(0)}, W_0=0,\\
&  \begin{cases}
     W_k=W_{k-1}+2a_k^NS_{k-1}, \quad k=1,2,\ldots,N-1,\\
     S_k=Z_k^{(0)}+2\tau S_{k-1}-S_{k-2},\\
  \end{cases}\\
&
Q_{N}^{(1)}(f,\tilde{\omega},x)=e^{i\tilde{\omega}}\left(W_{N-1}+S_{N-1}a_N^N+p_N(\tau)\dashint_{-1}^{1}\frac{e^{i\tilde{\omega}y}}{y-\tau}dy\right).
\end{align*}
The integral
$\dashint_{-1}^{1}\frac{e^{i\tilde{\omega}y}}{y-\tau}dy$ can be
evaluated from \eqref{eq:explicit form}.
\end{list}

The computational complexity is $\mathcal{O}(N\log_2N)$ operations.

\subsubsection{Computation of $I_2(x)$}
The computation of $I_2(x)$ is relatively easy, since the integrand
does not have singularity. Under the same assumptions as in Lemma
\ref{lem:transfer transform}, it is readily seen that
\begin{align}\label{eq:I2}
I_2(x)=\int_{a}^{\infty}\frac{f(t)}{t-x}e^{i\omega t}dt=ie^{i\omega
a}\int_{0}^{\infty}\frac{f(a+ip)}{a-x+ip}e^{-\omega p}dp.
\end{align}
As $a$ is an arbitrary real number larger than $x$, we may select it
so that the integrand of the last integral in \eqref{eq:I2} is well
behaved. Thus, this integral can be computed efficiently by
Gauss-Laguerre quadrature rule:
\begin{align}\label{def: GaussLaguerre2}
I_2(x)\simeq Q_n^{(2)}(f,\omega,x)=\frac{ie^{i\omega
a}}{\omega}\sum_{k=1}^{n}w_k\frac{f(a+\frac{it_k}{\omega})}{a-x+\frac{it_k}{\omega}},
\end{align}
where $\{t_k,w_k\}$ are the nodes and weights of the Gauss-Laguerre
quadrature associated with the weight $e^{-t}$.

\begin{example}
We apply quadrature rule $Q_n^{(2)}(f,\omega,x)$ in \eqref{def:
GaussLaguerre2} to calculate $I_2(x)$ with $f(t)=\sqrt{t}/(1+t)$ and
$x=0.02$. The error is presented in Figure \ref{fig:error of I_2}.
We can see that more accurate approximations are obtained as $a$
increases. For each fixed $a$, more accurate approximations are
obtained as $\omega$ increases.
\begin{figure}[h]
\centering
\begin{overpic}
[width=4.5cm]{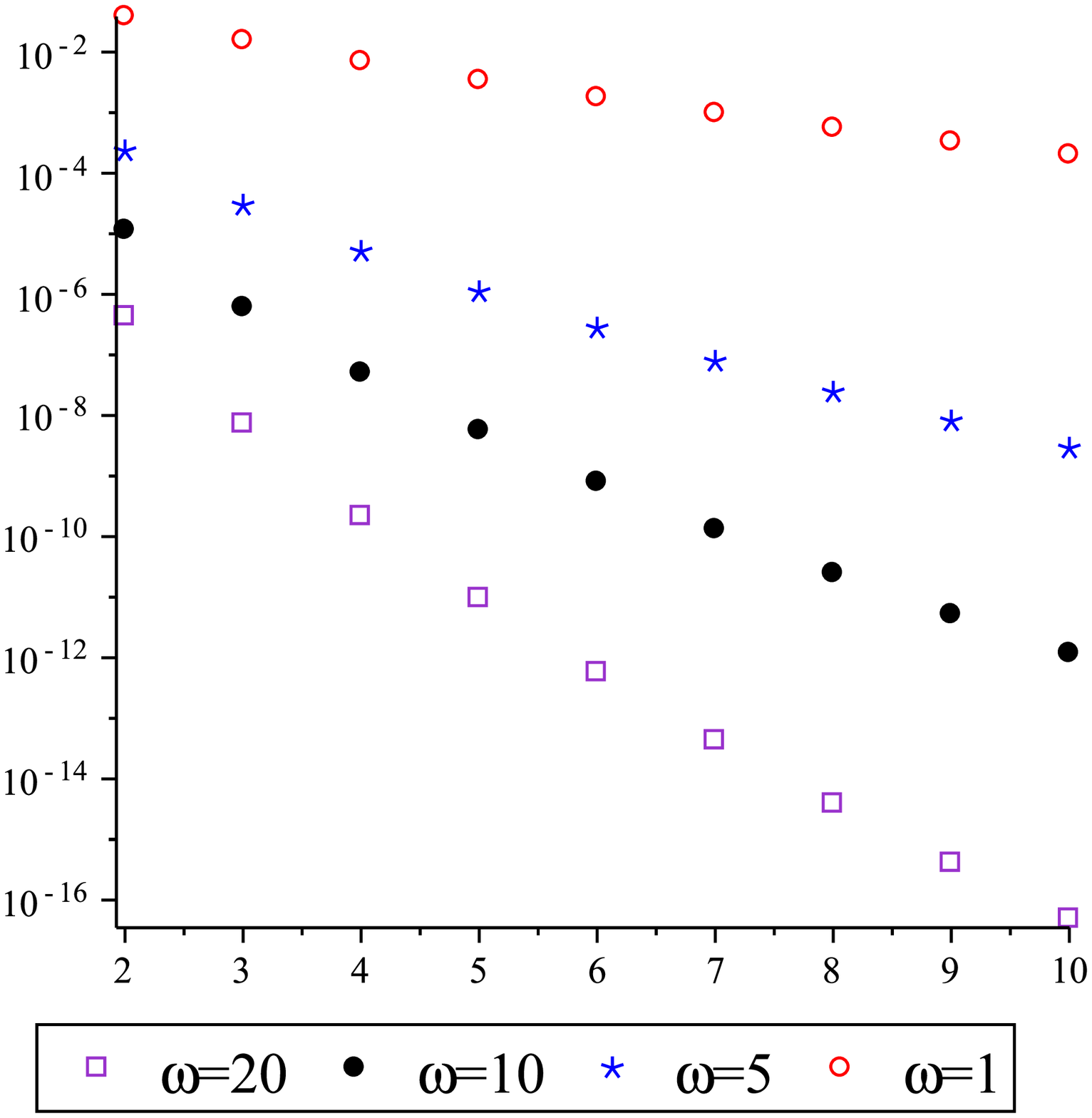}
\end{overpic}
\quad
\begin{overpic}
[width=4.5cm]{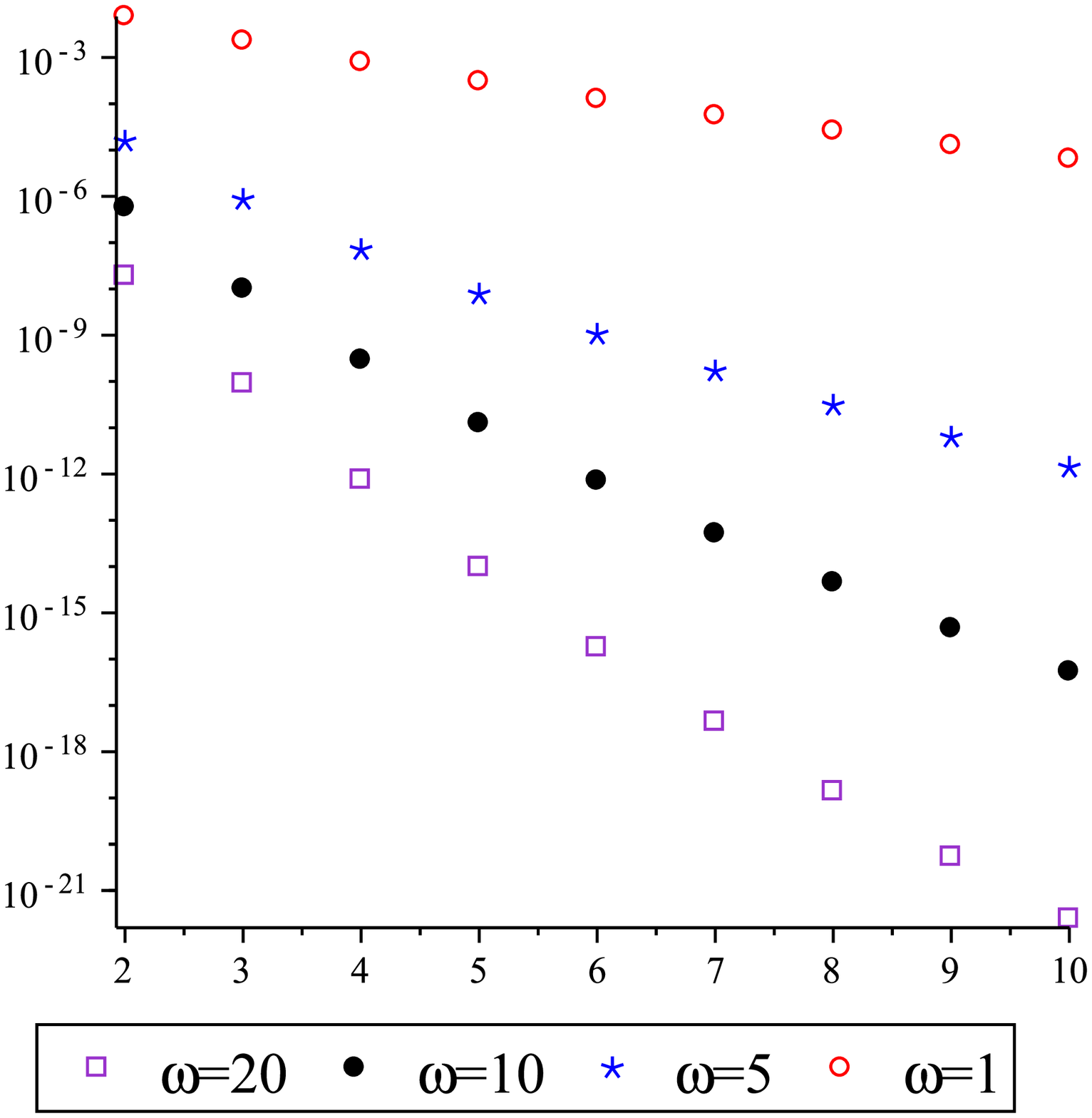}
\end{overpic}
\quad
\begin{overpic}
[width=4.5cm]{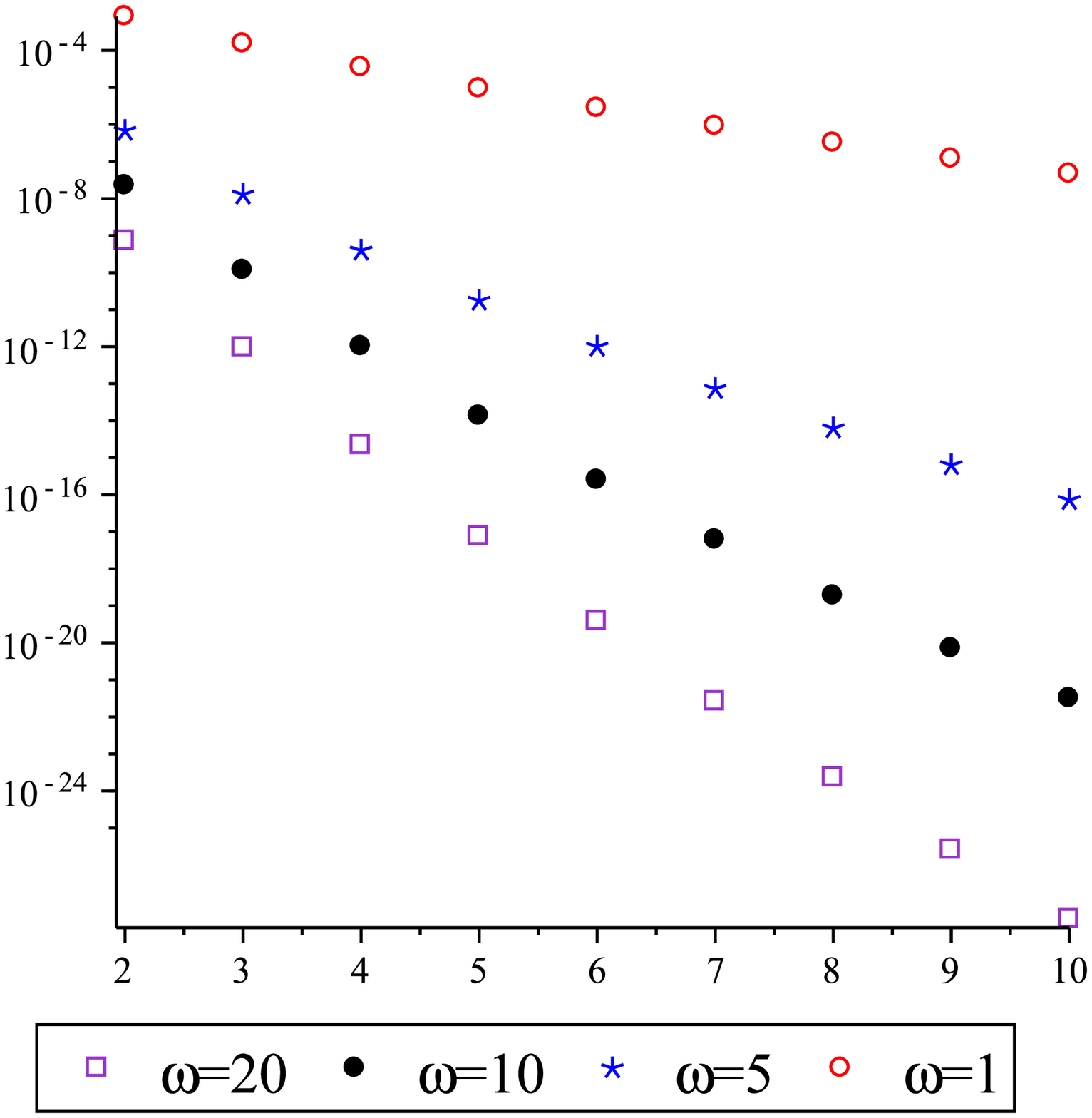}
\end{overpic}
\caption{The error of the quadrature $Q_n^{(2)}(f,\omega,x)$ for
$a=1$ (left), $a=2$ (middle) and $a=4$ (right) with
$f(t)=\sqrt{t}/(1+t)$ and $n$ ranging from $2$ to $10$. Here we
choose $x=0.02$.} \label{fig:error of I_2}
\end{figure}
\end{example}

\subsubsection{Numerical examples}
We present in this section some numerical experiments to illustrate
the efficiency of numerical methods presented in the above two
sections.

\begin{example} Let us consider the integral \eqref{eq:exp example} with $c=1$
and $x$ close to zero. The exact solution is given in \eqref{eq:
exact solution}. Since $\alpha=0$ in this case, we use Algorithm II
to compute $I_1(x)$, and the Gauss-Laguerre quadrature rule
$Q_n^{(2)}(f,\omega,x)$ defined in \eqref{def: GaussLaguerre2} to
compute $I_2(x)$. The absolute error for several values of $x$ is
presented in Table 1, which indicates the proposed method is
uniformly accurate as $x\rightarrow0$. The absolute error for
several values of $\omega$ with $a=1$ and $x=0.02$ is presented in
Table 2. We can see that the convergence is quite
rapid as $N$ and $n$ increase. 
\begin{table}[h] %
\caption{Absolute error in computing \eqref{eq:exp example} with
$c=1$, $a=1$, $\omega=10$ and $x=10^{-\delta}$.}\vspace{0.4cm}
\label{aggiungi}\centering%
\begin{tabular}{ccllll}
\toprule%
$n$  &   $N$    &  $\delta=1$     &   $\delta=2$   &  $\delta=3$  &
$\delta=4$
\\\toprule
4    & 4  &  1.22$\times 10^{-5}$    & 3.80$\times 10^{-5}$  & 3.52$\times 10^{-5} $ &3.42$\times 10^{-5}$  \\
     & 8  &  3.30$\times 10^{-7}$    & 1.83$\times 10^{-7}$  & 1.73$\times 10^{-7} $ &1.72$\times 10^{-7}$  \\
     & 16 &  3.30$\times 10^{-7}$    & 1.83$\times 10^{-7}$  & 1.73$\times 10^{-7} $ &1.72$\times 10^{-7}$  \\
8    & 4  &  1.19$\times 10^{-5}$    & 3.80$\times 10^{-5}$  & 3.52$\times 10^{-5} $ &3.41$\times 10^{-5}$  \\
     & 8  &  2.69$\times 10^{-10}$   & 1.09$\times 10^{-10}$ & 1.04$\times 10^{-10}$ &1.03$\times 10^{-10}$  \\
     & 16 &  2.86$\times 10^{-10}$   & 1.09$\times 10^{-10}$ & 9.96$\times 10^{-11}$ &9.87$\times 10^{-11}$  \\
16   & 4  &  1.19$\times 10^{-5}$    & 3.80$\times 10^{-5}$  & 3.52$\times 10^{-5} $ &3.41$\times 10^{-5}$  \\
     & 8  &  2.37$\times 10^{-11}$   & 2.65$\times 10^{-12}$ & 9.36$\times 10^{-12}$ &8.00$\times 10^{-12}$  \\
     & 16 &  1.30$\times 10^{-14}$   & 2.97$\times 10^{-15}$ & 2.58$\times 10^{-15}$ &2.54$\times 10^{-15}$  \\\bottomrule
\end{tabular}
\end{table}

\begin{table}[h] %
\caption{Absolute error in computing \eqref{eq:exp example} with
$c=1$, $a=1$, $x=0.02$ and several values of
$\omega$.}\vspace{0.4cm}
\label{aggiungi}\centering%
\begin{tabular}{ccllll}
\toprule%
$n$  &   $N$    &  $\omega=5$     &   $\omega=20$   &  $\omega=80$ &
$\omega=320$
\\\toprule
4    & 4  &  2.84$\times 10^{-5}$    & 2.30$\times 10^{-5}$   & 1.56$\times 10^{-5}  $ &1.73$\times 10^{-5} $  \\
     & 8  &  1.81$\times 10^{-5}$    & 8.92$\times 10^{-10}$  & 1.28$\times 10^{-11} $ &1.81$\times 10^{-11}$  \\
     & 16 &  1.81$\times 10^{-5}$    & 8.69$\times 10^{-10}$  & 4.89$\times 10^{-15} $ &1.92$\times 10^{-20}$  \\
8    & 4  &  1.65$\times 10^{-5}$    & 2.30$\times 10^{-5}$   & 1.56$\times 10^{-5}  $ &1.73$\times 10^{-5} $  \\
     & 8  &  8.16$\times 10^{-10}$   & 3.12$\times 10^{-11}$  & 1.28$\times 10^{-11} $ &1.81$\times 10^{-11}$  \\
     & 16 &  1.08$\times 10^{-10}$   & 2.00$\times 10^{-14}$  & 8.08$\times 10^{-24} $ &3.65$\times 10^{-25}$  \\
16   & 4  &  1.66$\times 10^{-5}$    & 2.30$\times 10^{-5}$   & 1.56$\times 10^{-5}  $ &1.73$\times 10^{-5} $  \\
     & 8  &  8.69$\times 10^{-11}$   & 3.12$\times 10^{-11}$  & 1.28$\times 10^{-11} $ &1.81$\times 10^{-11}$  \\
     & 16 &  7.49$\times 10^{-11}$   & 5.44$\times 10^{-21}$  & 2.78$\times 10^{-25} $ &3.65$\times 10^{-25}$  \\\bottomrule
\end{tabular}
\end{table}
\end{example}

\begin{example} We consider \eqref{eq:example 2rd} with $x$\ close to
the origin. Since $\alpha=1/2$ in this case, we use Algorithm I to
compute $I_1(x)$ and the Gauss-Laguerre quadrature rule
$Q_n^{(2)}(f,\omega,x)$ in \eqref{def: GaussLaguerre2} to compute
$I_2(x)$. For simplicity, we choose $N_1=2N$ when
$N>2\widetilde{\omega}$ in our implementation and the last integral
in \eqref{eq: closed form} is computed by using 32-point
Gauss-Laguerre quadrature. The absolute error is presented in Table
3 for $a=1$ and $\omega=10$, which implies that the convergence is
quite rapid as $N$ and $n$ increase. In Table 4 we present the
absolute error for several values of $\omega$ with fixed $a$ and
$x$. As we can see, the accuracy greatly improves as $\omega$
increases.
\begin{table}[h] %
\caption{Absolute error in computing \eqref{eq:example 2rd} with
$a=1$, $\omega=10$ and $x=10^{-\delta}$.}\vspace{0.4cm}
\label{aggiungi}\centering%
\begin{tabular}{ccllll}
\toprule%
$n$  &   $N$    &  $\delta=1$     &   $\delta=2$   &  $\delta=3$  &
$\delta=4$
\\\toprule
4    & 4  &  1.53$\times 10^{-3}$    & 3.56$\times 10^{-3}$  & 4.56$\times 10^{-3} $ &4.67$\times 10^{-3}$  \\
     & 8  &  1.31$\times 10^{-6}$    & 1.63$\times 10^{-7}$  & 2.87$\times 10^{-6} $ &3.25$\times 10^{-6}$  \\
     & 16 &  1.16$\times 10^{-7}$    & 4.51$\times 10^{-8}$  & 4.08$\times 10^{-8} $ &4.03$\times 10^{-8}$  \\
8    & 4  &  1.53$\times 10^{-3}$    & 3.56$\times 10^{-3}$  & 4.56$\times 10^{-3} $ &4.67$\times 10^{-3}$  \\
     & 8  &  1.23$\times 10^{-6}$    & 1.29$\times 10^{-7}$  & 2.89$\times 10^{-6} $ &3.27$\times 10^{-6}$  \\
     & 16 &  9.20$\times 10^{-11}$   & 2.22$\times 10^{-11}$ & 1.62$\times 10^{-11}$ &1.51$\times 10^{-11}$  \\
16   & 4  &  1.53$\times 10^{-3}$    & 3.56$\times 10^{-3}$  & 4.56$\times 10^{-3} $ &4.67$\times 10^{-3}$  \\
     & 8  &  1.22$\times 10^{-6}$    & 1.29$\times 10^{-7}$  & 2.89$\times 10^{-6} $ &3.27$\times 10^{-6}$  \\
     & 16 &  1.39$\times 10^{-12}$   & 2.46$\times 10^{-12}$ & 1.33$\times 10^{-12}$ &2.38$\times 10^{-12}$  \\\bottomrule
\end{tabular}
\end{table}

\begin{table}[h] %
\caption{Absolute error in computing \eqref{eq:example 2rd} with
$a=1$, $x=0.02$ and several values of $\omega$.}\vspace{0.4cm}
\label{aggiungi}\centering%
\begin{tabular}{ccllll}
\toprule%
$n$  &   $N$    &  $\omega=5$     &   $\omega=20$   &  $\omega=80$ &
$\omega=320$
\\\toprule
4    & 8  &  5.50$\times 10^{-6}$    & 2.02$\times 10^{-6}$   & 2.04$\times 10^{-6}  $ &2.38$\times 10^{-6}$  \\
     & 16 &  5.08$\times 10^{-6}$    & 2.15$\times 10^{-10}$  & 1.23$\times 10^{-12} $ &2.60$\times 10^{-12}$  \\
     & 32 &  5.08$\times 10^{-6}$    & 2.15$\times 10^{-10}$  & 1.19$\times 10^{-15} $ &4.65$\times 10^{-21}$  \\
8    & 8  &  2.06$\times 10^{-6}$    & 2.02$\times 10^{-6}$   & 2.04$\times 10^{-6}  $ &2.38$\times 10^{-6}$  \\
     & 16 &  2.41$\times 10^{-8}$    & 4.17$\times 10^{-13}$  & 1.23$\times 10^{-12} $ &2.60$\times 10^{-12}$  \\
     & 32 &  2.41$\times 10^{-8}$    & 3.86$\times 10^{-15}$  & 2.59$\times 10^{-24} $ &6.75$\times 10^{-25}$  \\
16   & 8  &  2.08$\times 10^{-6}$    & 2.02$\times 10^{-6}$   & 2.04$\times 10^{-6}  $ &2.38$\times 10^{-6}$  \\
     & 16 &  1.37$\times 10^{-11}$   & 4.13$\times 10^{-13}$  & 1.23$\times 10^{-12} $ &2.60$\times 10^{-12}$  \\
     & 32 &  1.39$\times 10^{-11}$   & 9.01$\times 10^{-22}$  & 1.11$\times 10^{-24} $ &6.75$\times 10^{-25}$  \\\bottomrule
\end{tabular}
\end{table}
\end{example}
\begin{remark}
When $\alpha=0$, the uniform convergence of the quadrature rule
\eqref{def: finiteoscill} has been proved in \cite{wang1}. When
$0<\alpha<1$, however, we don't currently know whether the uniform
convergence still holds. Numerical results show that the quadrature
rule \eqref{def: finiteoscill} remains accurate when $x$ is small,
as can be observed from Table 3.
\end{remark}

\subsection{The regime $x=0$}
When $x=0$, we need to deal with the Hadamard finite-part integral
$\ddashint_0^{\infty}e^{i\omega t}\frac{f(t)}{t}dt$, which is
introduced in Section \ref{sec:hadamard}. In view of
\eqref{eq:Hadamard int}, it suffices to find a fast method for
evaluating the integral
\begin{align}
\int_0^{\infty}\frac{f(t)-a_0t^{-\alpha}}{t}e^{i\omega t}dt,
\end{align}
or equivalently,
\begin{align}
\int_0^{\infty}t^{-\alpha}g(t)e^{i\omega t}dt,
\end{align}
where
\begin{equation}
g(t)=t^{\alpha}\frac{f(t)-a_0t^{-\alpha}}{t}.
\end{equation}
Since $g(t)$ is holomorphic in the first quadrant, which follows
from the assumptions in Lemma \ref{lem:transfer transform} and
Theorem \ref{thm:asy x>0}, we can apply Gaussian quadrature rules as
proposed by Wong \cite{wong2}:
\begin{align}\label{eq: Wong formula}
\int_0^{\infty}t^{-\alpha}g(t)e^{i\omega
t}dt=\frac{\exp((1-\alpha)\frac{i\pi}{2})}{\omega^{1-\alpha}}\sum_{k=1}^{n}w_kg\left(\frac{it_k}{\omega}\right)+E_n(g),
\end{align}
where $\{t_k,w_j\}_{j=1}^{n}$ are the nodes and weights of the
generalized Gauss-Laguerre quadrature with respect to weight
$t^{-\alpha}e^{-t}$. The reminder $E_n(g)$ is given by
\begin{align*}
E_n(g)=\frac{n!\Gamma(n-\alpha+1)}{(2n)!\omega^{2n-\alpha+1}}e^{(2n-\alpha+1)\frac{i\pi}{2}}g^{(2n)}(i\xi/\omega),
\quad 0<\xi<\infty.
\end{align*}
Therefore, we can approximate $\ddashint_0^{\infty}e^{i\omega
t}\frac{f(t)}{t}dt$ by
\begin{equation}\label{def: Wong quadrature}
\left\{
                    \begin{array}{ll}
                      {\displaystyle \frac{e^{\frac{\pi}{2}(2-\alpha)i}\omega^\alpha}{\alpha}\Gamma(1-\alpha)a_0+
                        \frac{\exp((1-\alpha)\frac{i\pi}{2})}{\omega^{1-\alpha}}\sum_{k=1}^{n}w_kg\left(\frac{it_k}{\omega}\right)},
                      & \hbox{if $0<\alpha<1$,} \\
                      {\displaystyle ( i\frac{\pi}{2}-\gamma-\log\omega)f(0)+\sum_{k=1}^{n}w_k\frac{f\left(\frac{it_k}{\omega}\right)-f(0)}{t_k}}, & \hbox{if $\alpha=0$,}
                    \end{array}
                  \right.
\end{equation}

\begin{example}
We use quadrature rule \eqref{def: Wong quadrature} to approximate
\eqref{eq:Hadamard int} with $f(t)=e^{-t}$ and
$f(t)=\sqrt{t}/(1+t)$, which corresponds to $\alpha=0$ and
$\alpha=\frac{1}{2}$ respectively, for several values of $\omega$.
The absolute errors are shown in Figure \ref{fig:hadarmard}.
Clearly, the convergence of the quadrature rule \eqref{def: Wong
quadrature} is rapid and satisfactory.

\begin{figure}[h]
\centering
\begin{overpic}
[width=5.5cm]{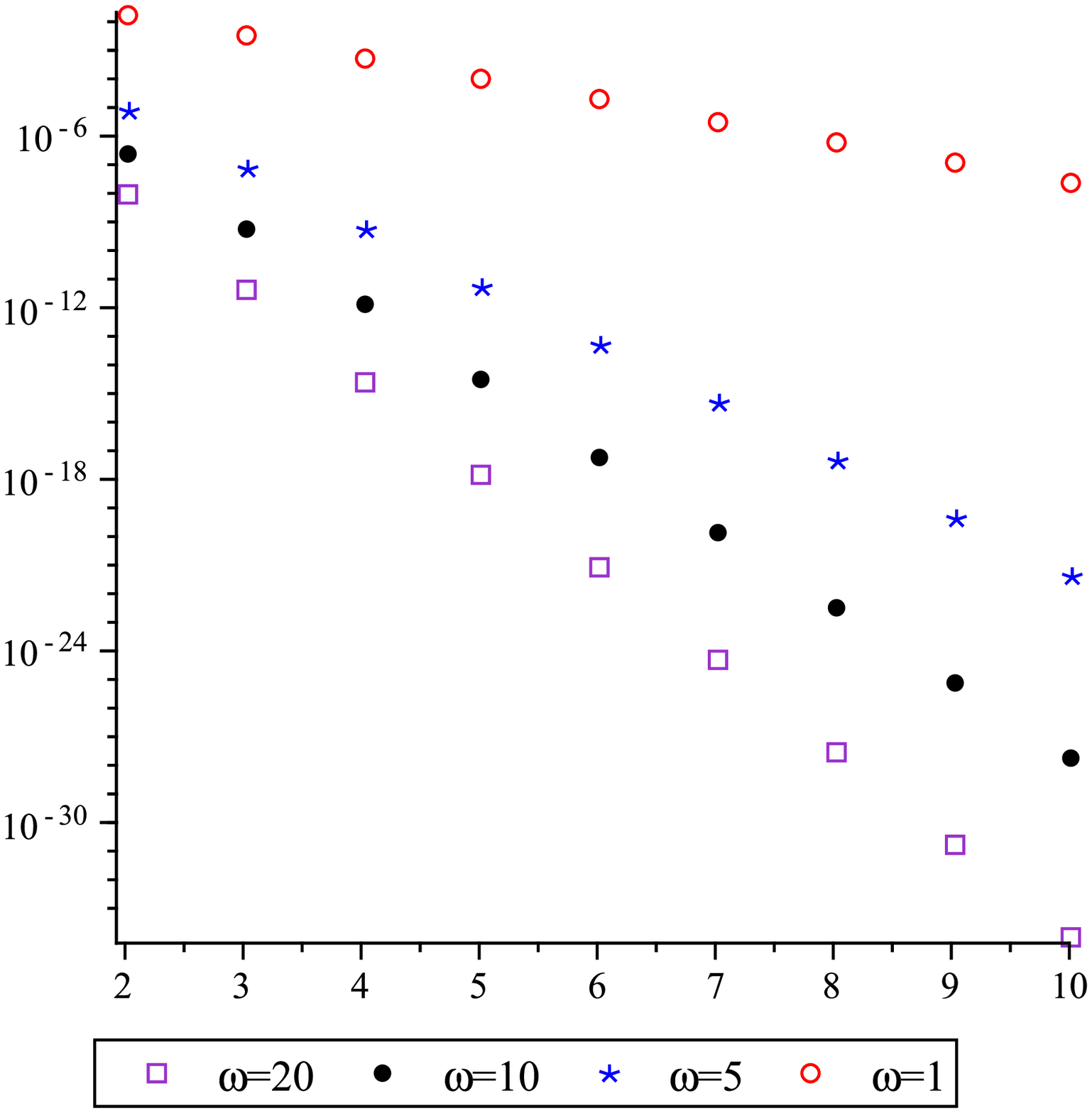}
\end{overpic}
\qquad
\begin{overpic}
[width=5.5cm]{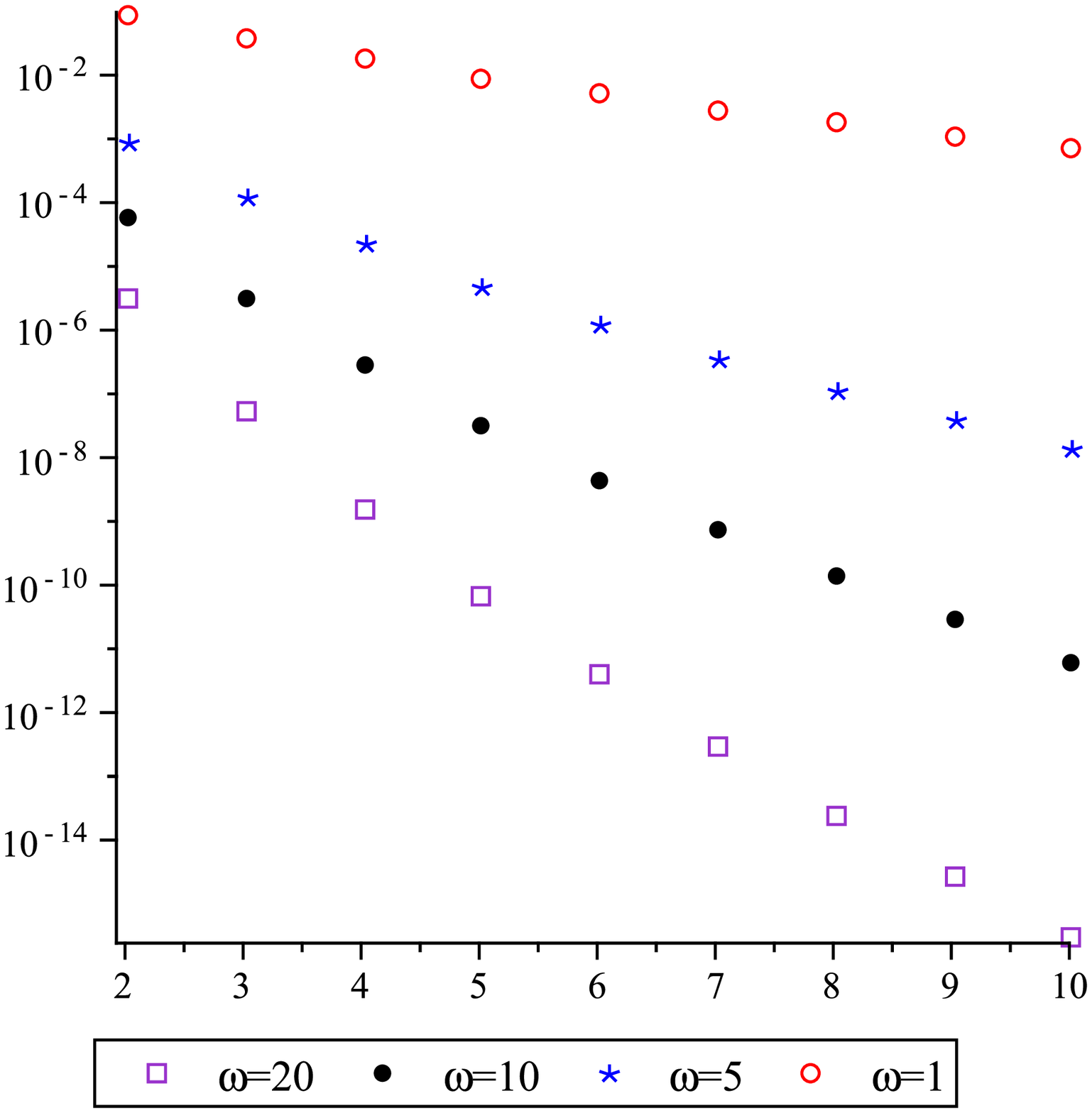}
\end{overpic}
\caption{The error of the quadrature \eqref{def: Wong quadrature}
for $f(t)=e^{-t}$ (left) and $f(t)=\sqrt{t}/(1+t)$ (right) with $n$
ranging from $2$ to $10$.}\label{fig:hadarmard}
\end{figure}

\end{example}

\section{Extensions}\label{sec:extension}
Oscillatory Hilbert transforms with Bessel type oscillators are
defined by
\begin{equation*}
H^+(f(t)J_\nu(\omega
t))(x):=\dashint_{0}^{\infty}\frac{f(t)}{t-x}J_\nu(\omega t)dt
\end{equation*}
and
\begin{equation*}
H^+(f(t)Y_\nu(\omega
t))(x):=\dashint_{0}^{\infty}\frac{f(t)}{t-x}Y_\nu(\omega t)dt,
\end{equation*}
where $J_{\nu}(t)$ and $Y_{\nu}(t)$ are the Bessel functions of the
first and second kind, respectively. Such kind of transformations
have applications in physics such as water-wave radiation problem
\cite{martin}. It turns out that Lemma \ref{lem:transfer transform}
can also be extended to oscillatory Bessel Hilbert transforms, by
using similar ideas.
\begin{lemma}\label{lem:Bessel}
Suppose that $f$ has an analytic continuation to the right half
plane, except possibly a branch point at the origin, and there exist
constants $M>0$, $\delta < \frac{3}{2}$ and $0\leq d < \omega$ such
that
\begin{equation}
|f(z)|\leq M|z|^{\delta}e^{d \Im(z)},
\end{equation}\label{eq:growth of f Bessel}
as $|z|\rightarrow\infty$ in the right-half plane. Then,
\begin{align}\label{def: Bessel first kind}
H^+(f(t)J_\nu(\omega t))(x)=-\pi f(x)Y_\nu(\omega
x)-\frac{1}{\pi}\int_{0}^{\infty}\frac{K_\nu(\omega
y)}{y^2+x^2}g_1(y)dy
\end{align}
and
\begin{align}\label{def: Bessel second kind}
H^+(f(t)Y_\nu(\omega t))(x)=\pi f(x)J_\nu(\omega
x)+\frac{i}{\pi}\int_{0}^{\infty}\frac{K_\nu(\omega
y)}{y^2+x^2}g_2(y)dy,
\end{align}
whenever the integrals exist, where $K_\nu(t)$ is the modified
Bessel function of the second kind and
\[g_j(y)=(x+iy)e^{-\frac{\nu}{2}\pi i}f(iy)+(-1)^{j+1}(x-iy)e^{\frac{\nu}{2}\pi
i}f(-iy), \quad j=1,2.
\]
\end{lemma}
\begin{proof}
We only give a sketched proof of \eqref{def: Bessel first kind},
since the proof of \eqref{def: Bessel second kind} can be handled in
a similar manner. On account of the identity (see \cite[Eq.~
9.6.4]{Abram})
\[
i\pi J_\nu(z)=e^{-\frac{\nu}{2}\pi i}K_\nu(ze^{-\frac{1}{2}\pi
i})-e^{\frac{\nu}{2}\pi i}K_\nu(ze^{\frac{1}{2}\pi i}),\qquad |\arg
z|\leq \frac{\pi}{2},
\]
it follows
\begin{align}\label{eq: bessel equation1}
&\dashint_{0}^{\infty}\frac{f(t)}{t-x}J_\nu(\omega t)dt\nonumber
\\
&=\frac{1}{i\pi}\left[e^{-\frac{\nu}{2}\pi
i}\dashint_{0}^{\infty}\frac{f(t)}{t-x}K_\nu(-i\omega
t)dt-e^{\frac{\nu}{2}\pi
i}\dashint_{0}^{\infty}\frac{f(t)}{t-x}K_\nu(i\omega t)dt \right].
\end{align}
For the first integral on the right hand side of \eqref{eq: bessel
equation1}, by considering the same contour as shown in Figure
\ref{fig:integralcontour}, we obtain from Cauchy's theorem that
\begin{align}\label{eq:1st integral}
&\dashint_{0}^{R}\frac{f(t)}{t-x}K_\nu(-i\omega t)dt \nonumber \\
&=i\pi f(x)K_\nu(-i\omega
x)-\int_{\Gamma_R}\frac{f(z)}{z-x}K_\nu(-i\omega
z)dz+i\int_{0}^{R}\frac{f(iy)}{iy-x}K_\nu(\omega y)dy.
\end{align}
Recall the asymptotic expansion of the modified Bessel function
$K_\nu(z)$ \cite[p.~378]{Abram}
\[
K_\nu(z)\sim
\sqrt{\frac{\pi}{2z}}e^{-z}\left[1+\frac{4\nu^2-1}{8z}+\mathcal{O}(z^{-2})\right],
\quad z\rightarrow\infty, \quad |\arg z|< \frac{3\pi}{2}.
\]
We have the following estimation of the integral over $\Gamma_R$ for
large $R$
\[\begin{array}{lll} {\displaystyle
\left|\int_{\Gamma}\frac{f(z)}{z-x}K_\nu(-i\omega
z)dz\right|}&=&{\displaystyle
\left|\int_{0}^{\frac{\pi}{2}}\frac{f(R e^{i\theta})}{R
e^{i\theta}-x}K_\nu(-i\omega R e^{i\theta})iR
e^{i\theta}d\theta\right|}\\
&\leq&{\displaystyle \frac{MR^{1+\delta}}{R-x}
\int_{0}^{\frac{\pi}{2}}e^{dR\sin\theta}|K_\nu(-i\omega R e^{i\theta})|d\theta}\\
&\leq&{\displaystyle
\frac{MR^{1+\delta}}{R-x}\sqrt{\frac{\pi}{2\omega
R}}\int_{0}^{\frac{\pi}{2}}e^{-(\omega-d) R
\sin\theta}(1+\mathcal{O}(R^{-1}))d\theta }\\
&\rightarrow& 0, \qquad \textrm{as $R\to\infty$}.
\end{array}\]
Hence, letting $R\rightarrow\infty$ in \eqref{eq:1st integral}, we
arrive at
\begin{equation}\label{eq:1st integral fin}
\dashint_{0}^{\infty}\frac{f(t)}{t-x}K_\nu(-i\omega t)dt=i\pi
f(x)K_\nu(-i\omega
x)+i\int_{0}^{\infty}\frac{f(iy)}{iy-x}K_\nu(\omega y)dy.
\end{equation}
Similarly, we can deform the integration path to the negative
imaginary axis for the second integral on the right hand side of
\eqref{eq: bessel equation1} and get
\begin{equation}\label{eq:2nd integral fin}
\dashint_{0}^{\infty}\frac{f(t)}{t-x}K_\nu(i\omega t)dt=-i\pi
f(x)K_\nu(i\omega
x)+i\int_{0}^{\infty}\frac{f(-iy)}{iy+x}K_\nu(\omega y)dy.
\end{equation}
A combination of \eqref{eq: bessel equation1}, \eqref{eq:1st
integral fin} and \eqref{eq:2nd integral fin} gives us \eqref{def:
Bessel first kind}.
\end{proof}

From the above lemma, we establish several interesting identities
with $\nu=0,1$, which have not been found in classical reference
books \cite{Abram,Gradshteyn} and which might have important
applications in practice:
\begin{corollary}We have
\begin{align}\label{eq:Bessel1}
\dashint_{0}^{\infty}\frac{J_{\nu}(\omega
t)}{t-x}dt=(-1)^{\nu+1}\frac{\pi}{2}
\left[\mathrm{\mathbf{H}}_{-\nu}(\omega x)+(-1)^\nu Y_{\nu}(\omega
x)\right],\quad \nu=0,1,
\end{align}
and
\begin{align}\label{eq:Bessel2}
\dashint_{0}^{\infty}\frac{Y_0(\omega t)}{t-x}dt=\pi J_0(\omega
x)-\frac{2}{\pi}S_{-1,0}(\omega x),
\end{align}
and
\begin{align}\label{eq:Bessel21}
\dashint_{0}^{\infty}\frac{tY_1(\omega t)}{t-x}dt=\pi xJ_1(\omega
x)-\frac{4x}{\pi}S_{-2,1}(\omega x),
\end{align}
where $\mathbf{H}_{\nu}(z)$ is the Struve function and
$S_{\mu,\nu}(z)$ is the Lommel function of the second kind.
\end{corollary}
\begin{proof}
With $f=1$ and $\nu=0,1$ in \eqref{def: Bessel first kind}, it
follows that
\begin{align}\label{eq:Bessel11}
\dashint_{0}^{\infty}\frac{J_\nu(\omega t)}{t-x}dt=\left\{
                                                     \begin{array}{ll}
                                                       -\pi Y_0(\omega
x)-\frac{2x}{\pi}\int_{0}^{\infty}\frac{K_0(\omega y)}{y^2+x^2}dy, & \hbox{if $\nu=0$,} \\
                                                       -\pi Y_1(\omega
x)-\frac{2}{\pi}\int_{0}^{\infty}\frac{yK_1(\omega y)}{y^2+x^2}dy, &
\hbox{if $\nu=1$.}
                                                     \end{array}
                                                   \right.
\end{align}
Recall the identity (see \cite[Eq. 6.566.3]{Gradshteyn})
\begin{align*}
\int_{0}^{\infty}\frac{y^{\nu}K_{\nu}(\omega
y)}{y^2+x^2}dy=\frac{\pi^2x^{\nu-1}}{4\cos(\nu\pi)}[\mathrm{\mathbf{H}}_{-\nu}(\omega
x)-Y_{-\nu}(\omega x)],\quad \omega>0,\quad \Re x>0,\quad \Re
\nu>-\frac{1}{2}.
\end{align*}
Formula \eqref{eq:Bessel1} then follows from inserting the above
identity with $\nu=0$ and $\nu=1$, respectively.

To show \eqref{eq:Bessel2}, we set $f=1$ and $\nu=0$ in \eqref{def:
Bessel second kind} and obtain
\begin{align}\label{eq:zero order Bessel second}
\dashint_{0}^{\infty}\frac{Y_0(\omega t)}{t-x}dt=\pi J_0(\omega
x)-\frac{2}{\pi}\int_{0}^{\infty}\frac{yK_0(\omega y)}{y^2+x^2}dy.
\end{align}
Since
\begin{align}
\int_{0}^{\infty}y^{1+\nu}(y^2+x^2)^{\mu}K_{\nu}(\omega
y)dy=2^{\nu}\Gamma(\nu+1)x^{\nu+\mu+1}\omega^{-1-\mu}S_{\mu-\nu,\mu+\nu+1}(\omega
x)
\end{align}
for $\Re x>0$, $\Re \omega>0$ and $\Re \nu>-1$ (see
\cite[Eq.~6.565.7]{Gradshteyn}), substituting the above identity
with $\nu=0$ and $\mu=-1$ into \eqref{eq:zero order Bessel second}
gives us \eqref{eq:Bessel2}.

Finally, the proof of \eqref{eq:Bessel21} is similar to that of
\eqref{eq:Bessel2}, and we omit the details here.
\end{proof}

We expect that Lemma \ref{lem:Bessel} might play an important role
in the asymptotic and numerical study of oscillatory Hilbert
transforms. Note that the modified Bessel functions $K_{\nu}$ have a
non-integrable singularity at $0$ for $\nu\geq1$.

Finally, we like to recall that the numerical method proposed in
\cite{asheim} generalized steepest descent-based methods for
Fourier-type integrals of the form \eqref{eq:Fourier integrals} to
oscillatory transforms of the form
\[ \int_{0}^{\infty}f(x)H(\omega x)dx,  \]
where $H$ can be, e.g., a Bessel function. We expect that this
method can be extended in turn to compute oscillatory Hilbert
transforms with more general oscillators as well.


\section{Concluding remarks}\label{sec:remarks}
In this paper, we have considered asymptotic expansions and fast
computation of the oscillatory Hilbert transforms
\eqref{def:one-side transform}. Unlike previous work, which focused
on the behaviour of oscillatory Hilbert transforms for large $x$, we
derive asymptotic expansions of such transforms for large $\omega$.
These expansions clarify the asymptotic behaviour for large $\omega$
and provide a powerful approach for designing efficient and accurate
approximation methods.

Numerical methods for the calculation of the oscillatory Hilbert
transforms are presented. We classify our discussion into three
regimes, namely, $x=\mathcal{O}(1)$ or $x\gg1$, $0<x\ll 1$ and
$x=0$. For each regime, we have designed efficient numerical
approaches. Even for small values of $\omega$, our proposed methods
remain quite accurate. Numerical examples are provided to confirm
our analysis.

In the implementation of our methods, the function $f(x)$ is
required to be analytic in the first quadrant of the complex plane.
It seems that this requirement is more restrictive. If $f(x)$ is a
differentiable but not analytic function, then we may construct a
suitable Filon-type method which utilizes Hermite interpolation to
compute such transforms. These results will be reported in our
future work.

\section*{Acknowledgement}
We thank Andreas Asheim for helpful discussions.


\baselineskip 0.50cm
\begin{thebibliography}{99}
\bibitem{Ablowitz}
M. J. Ablowitz and A. S. Fokas, \emph{Complex Variables:
Introduction and Applications}, Cambridge University Press, 2003.


\bibitem{Abram}
M. Abramowitz and I. A. Stegun, \emph{Handbook of Mathematical
Functions}, National Bureau of Standards, Washington, D.C., 1964.

\bibitem{asheim}
A. Asheim and D. Huybrechs, \emph{Complex Gaussian quadrature for
oscillatory integral transforms}, Report TW 594, 2011.


\bibitem{Berrut}
J. P. Berrut and L. N. Trefethen, \emph{Barycentric Lagrange
interpolation}, SIAM Review, 46 (2004), 501-517.

\bibitem{capobianco}
M. R. Capobianco, G. Criscuolo, \emph{On quadrature for Cauchy
principal value integrals of oscillatory functions}, J. Comput.
Appl. Math., 156 (2003), 471-486.

\bibitem {Clenshaw}
C. W. Clenshaw and A. R. Curtis, \emph{A method for numerical
integration on an automatic computer}, Numer. Math., 2 (1960),
197-205.


\bibitem{deano}
A. Dea\~{n}o and D. Huybrechs, \emph{Complex Gaussian quadrature of
oscillatory integrals}, Numer. Math., 112 (2009), 197-219.

\bibitem{dominguez}
V. Dom\'{i}nguez, I. G. Graham and V. P. Smyshlyaev, \emph{Stability
and error estimates for Filon-Clenshaw-Curtis rules for highly
oscillatory integrals}, IMA J. Numer. Anal., 31 (2011), 1253-1280.

\bibitem{evans}
K. C. Chung, G. A. Evans and J. R. Webster, \emph{A method to
generate generalized quadrature rules for oscillatory integrals},
Appl. Numer. Math., 34 (2000), 85-93.

\bibitem{Davis}
P. J. Davis and P. Rabinowitz, \emph{Methods of Numerical
Integration}, Second Edition, Academic Press, 1984.

\bibitem{gentleman}
W. M. Gentleman, \emph{Implementing Clenshaw-Curtis quadrature.} II,
Comm. ACM, 15 (1972), 343-346.

\bibitem{Gradshteyn}
I. S. Gradshteyn, I. M. Ryzhik, \emph{Tables of Integrals, Series,
and Products}, 6th ed. San Diego, CA: Academic Press, 2000.

\bibitem{hasegawa}
T. Hasegawa and T. Torii, \emph{An automatic quadrature for Cauchy
principal value integrals}, Math. Comp., 56 (1991), 741-754.

\bibitem{Huybrechs}
D. Huybrechs and S. Vandewalle, \emph{On the evaluation of highly
oscillatory integrals by analytic continuation}, SIAM J. Numer.
Anal., 44 (2006), 1026-1048.

\bibitem{Iserles}
A. Iserles and S. P. N{\o}rsett, \emph{Efficient quadrature of
highly oscillatory integrals using derivatives}, Proc. Royal. Soc.
A, 461 (2005), 1383-1399.

\bibitem{Iserles1}
A. Iserles and S. P. N{\o}rsett, \emph{On quadrature methods for
highly oscillatory integrals ans their implementation}, BIT, 44
(2004), 755-772.

\bibitem{king book}
F. W. King, \emph{Hilbert transforms: Volume 1}, Cambridge
University Press, 2009.

\bibitem{king}
F. W. King, G. J. Smethells, G. T. Helleloid and P. J. Pelzl,
\emph{Numerical evaluation of Hilbert transforms for oscillatory
functions: A convergence accelerator approach}, Comput. Phys. Comm.,
145 (2002), 256-266.

\bibitem{Krommer}
A. R. Krommer and C. W. Ueberhuber, \emph{Computational
Integration}, SIAM, Philadelphia, 1998.

\bibitem{levin}
D. Levin, \emph{Fast integration of rapidly oscillatory functions},
J. Comput. Appl. Math., 67 (1996), 95-101.

\bibitem{levin1}
D. Levin, \emph{Procedures for computing one- and two-dimensional
integrals of functions with rapid irregular oscillations}, Math.
Comp., 38 (1982), 531-538.

\bibitem{lyness}
J. N. Lyness, \emph{The Euler Maclaurin expansion for the Cauchy
principal value integral}, Numer. Math., 46 (1985), 611-622.

\bibitem{Mil}
G. V. Milovanovi\'{c}, \emph{Numerical calculation of integrals
involving oscillatory and singular kernels and some applications of
quadratures}, Comput. Math. Appl., 36 (1998), 19-39.

\bibitem{martin}
P. A. Martin, \emph{On the null-field equations for water-wave
radiation problems}, J. Fluid Mech., 113 (1981), 315-332.

\bibitem{monegato}
G. Monegato and J. N. Lyness, \emph{The Euler-Maclaurin expansion
and finite-part integrals}, Numer. Math., 81 (1998), 273-291.

\bibitem{Okecha1}
G. E. Okecha, \emph{Quadrature formulae for Cauchy principal value
integrals of oscillatory kind}, Math. Comp., 49 (1987), 259-268.

\bibitem{oliver}
J. Oliver, \emph{Relative error propagation in the recursive
solution of linear recurrence relations}, Numer. Math., 9 (1967),
323-340.

\bibitem{Olver 11}
S. Olver, \emph{Computing the Hilbert transform and its inverse},
Math. Comp., 80 (2011), 1745-1767.

\bibitem{Olver}
S. Olver, \emph{Moment-free numerical integration of highly
oscillatory functions}, IMA. J. Numer. Anal., 26 (2006), 213-227.

\bibitem{Olver1}
S. Olver, \emph{GMRES for the differentiation operator}, SIAM. J.
Numer. Anal., 47 (2009), 3359-3373.

\bibitem{piessens}
R. Piessens and M. Branders, \emph{On the computation of Fourier
transforms of singular functions}, J. Comput. Appl. Math., 43
(1992), 159-169.

\bibitem{ursell}
F. Ursell, \emph{Integrals with a large parameter: Hilbert
transforms}, Math. Proc. Camb. Soc., 93 (1983), 141-149.

\bibitem{wang1}
H. Wang and S. Xiang, \emph{Uniform approximations to Cauchy
principal value integrals of oscillatory functions}, Appl. Math.
Comp., 215 (2009), 1886-1894.

\bibitem{wang2}
H. Wang and S. Xiang, \emph{On the evaluation of Cauchy principal
value integrals of oscillatory functions}, J. Comput. Appl. Math.,
234 (2010), 95-100.


\bibitem{wong}
R. Wong, \emph{Asymptotic expansion of the Hilbert transform}, SIAM
J. Math. Anal., 11 (1980), 92-99.

\bibitem{wong2}
R. Wong, \emph{Quadrature formulas for oscillatory integral
transforms}, Numer. Math., 39 (1982), 351-360.

\bibitem{wong1}
R. Wong, \emph{Asymptotic Approximations of Integrals}, SIAM,
Philadelphia, 2001.

\bibitem{Xiang}
S. Xiang, \emph{Efficient Filon-type methods for
$\int_a^bf(x)e^{i\omega g(x)}dx$}, Numer. Math., 105 (2007),
633-658.

\bibitem{Xiang1}
S. Xiang, X. Chen and H. Wang, \emph{Error bounds for approximation
in Chebyshev points}, Numer. Math., 116 (2010), 463-491.
\end{thebibliography}
\end{document}